\documentclass[11pt]{article}
\usepackage[margin=1in]{geometry}
\usepackage{relsize}
\usepackage{amssymb}
\usepackage{amsmath}
\usepackage{enumitem}
\usepackage{graphicx}
\usepackage[linktocpage]{hyperref}
\usepackage{amsthm}
\usepackage{comment}
\usepackage{subfig}
\usepackage{wrapfig}
\usepackage{color}
\usepackage{bbm}
\usepackage{thmtools,thm-restate}
\usepackage[maxbibnames=99, backend=biber,style=alphabetic]{biblatex}
\bibliography{references} 
\newcommand{\Vol}{\text{Vol}}
\newcommand{\polylog}{\text{polylog}}
\DeclareMathOperator*{\argmax}{argmax}

\DeclareMathOperator{\Var}{Var}

\newtheorem{theorem}{Theorem}[section]

\newtheorem{lemma}{Lemma}[theorem]
\newtheorem{corollary}[theorem]{Corollary}
\newtheorem{claim}{Claim}[theorem]

\newtheorem{definition}{Definition}[section]

\begin{document}

\title{The $k$-Cap Process on Geometric Random Graphs}
\author{ Mirabel Reid\\ mreid48@gatech.edu\\ Georgia Tech \and Santosh S. Vempala\\ vempala@gatech.edu\\ Georgia Tech}
\date{}

\maketitle

\begin{abstract}
The $k$-cap (or $k$-winners-take-all) process on a graph works as follows: in each iteration, a subset of $k$ vertices of the graph are identified as winners; the next round winners are the vertices that have the highest total degree from the current winners, with ties broken randomly. This natural process is a simple model of firing activity and inhibition in the brain and has been found to have desirable robustness properties as an activation function. We study its convergence on directed geometric random graphs in any constant dimension, revealing rather surprising behavior, with the support of the current active set converging to lie in a small ball and the active set itself remaining essentially random within that.    
\end{abstract}

\tableofcontents
\thispagestyle{empty}
\newpage
\pagenumbering{arabic}

\section{Introduction}
	The function $k$-cap, also known as $k$-winners take all, takes $n$ real-valued elements and selects the $k$ elements with the highest values, i.e., it assigns $1$ to those $k$ and $0$ to all others, breaking ties randomly. It has found applications in machine learning, image processing and related fields~\cite{maass2000computational, wang2010analysis, xiao2019enhancing}. It has been proven to be computationally powerful; circuits employing $k$-WTA gates as nonlinearities can approximate arbitrary Boolean functions~\cite{maass2000computational}. Additionally, since the gradient of the $k$-WTA function is undefined at key points, it has been recently proposed as a technique for defending against attacks which use the gradient of a neural network to generate adversarial examples~\cite{xiao2019enhancing}. This property differentiates it from typical activation functions, such as ReLU or tanh. The $k$-cap process has also been proposed as a model of neural firing behavior, a motivation that we will discuss in more detail presently. 
	
	We study the $k$-cap process, which repeatedly applies $k$-cap to the  degrees of a random graph: at each time step $t>0$, the set $A_t$ consists of the $k$ vertices with the highest degree in $A_{t-1}$ (with ties broken randomly). Given this process, some natural questions arise: 
\begin{itemize}
\item  What does the $k$-cap process ``converge" to? 

\item When the process does converge, how quickly does it do so? 
\end{itemize}
To understand these and related questions, we investigate how the $k$-cap $A_t$ at time $t$, evolves as $t\rightarrow \infty$. If $G$ is the complete graph, every vertex has degree $k$ from $A_t$; so, assuming random tie-breaking, all vertices fire with probability $k/n$ at each time step. On the other hand, if $G$ is a sparse graph with a planted $k$-clique $H$, we expect $A_t = H$ to be a fixed point (for $k$ sufficiently large).  These examples indicate that the answers to the above questions depend on the graph structure; there are many classes of random graphs in which no meaningful convergence is expected. For example, we don't expect this process to converge on a directed Erdős–Rényi model of random graphs. Since this early influential model, many interesting variants of random graphs have been proposed, e.g., Power Law Random graphs, Stochastic Block Models, Geometric Random graphs, etc. The last of these seems particularly well-motivated for studying the $k$-cap process. In geometric random graphs, each vertex is assigned a position in a hidden variable space (for example, the cube $[0, 1]^d$). The probability that an edge $\mathbbm{1}_{(x, y)}$ exists in the graph is a function of the hidden variables of the endpoints. By using an edge probability function which decreases with distance in the hidden variable space, this creates subgraphs which are dense and concentrated within a small diameter subset of the space. The hidden variables can correspond to spatial distance, or they can represent similarity in a wider set of features. For example, the geometric random graph model has been used for social networks, where the hidden variable represents a closeness in ``social space" rather than physical distance~\cite{boguna2004models}.  This model has also been studied in the context of transportation networks, communication networks, and networks of neurons~\cite{bullmore2009complex, barthelemy2011spatial}. 

Properties of geometric random graphs have been thoroughly explored; see~\cite{penrose2003random} for comprehensive exposition. In the most common variant of the model, all vertices are placed in a $d$-dimensional space according to some distribution. If the distance between two vertices is less than $r$ (where $r$ is a parameter of the model), they are connected by an edge; otherwise, they are not. We study a \textit{directed, soft} geometric random graph where the edge probability decays exponentially with squared distance, i.e., the Gaussian kernel. This alternative model introduces asymmetry as well as long-range connections, both of which are important for real-life networks.

\paragraph{Motivation from the Brain. } \label{sec:brain-motivation}The network of the brain, called the connectome, is modeled as a sparse directed graph whose nodes represent neurons and whose directed edges represent synaptic connections. It is useful to view the connectome as consisting of many directed subgraphs (also called brain areas) with some connectivity between them. Neurons fire based on the total (weighted) input they sense from other neurons that are currently firing. 
An important and longstanding idea in neuroscience is that of an {\em assembly} of neurons --- a subset of densely interconnected nodes within a brain area which tend to fire together in response to the same input to the brain area~\cite{papadimitriou2020brain,buzsaki2019brain}. Assemblies are created through projection, where an outside stimulus fires (repeatedly), activating a subset of neurons. Two ideas, rooted in experimental findings in neuroscience, lead to the convergence of assembly projection in a random brain graph. The first is inhibition: at each step, the $k$ neurons with the highest total synaptic input are chosen to fire, while the rest are suppressed. The second is plasticity: if a neuron fires immediately following one of its pre-synaptic neighbors, the weight of the edge between them is increased. This causes neurons that `fire together' to `wire together', and strengthens internal connections each time an assembly is activated.
	
Rigorous analysis of the assembly model has thus far been based on a directed Erd\"{o}s-R\'{e}nyi random graph, where each pair of neurons has an equal probability of being connected via a synapse. There are two important ways in which this model departs from observed reality. First, the locations of neurons in the brain and the physical distance between them have a significant impact on the probability of connection. Long axons come with a cost in both material and energy, so neurons tend to prefer to create connections that are close in physical space. The principle of conservation of axonal wiring costs was proposed by Ramon y Cajal in the early 20th century~\cite{ramon1911histology}, and the relationship between distance and connection probability has been confirmed empirically~\cite{bullmore2009complex, cuntz2010one}. Moreover, models that take locality into account are better able to explain statistical deviations of the connectome from the standard random graph model, as observed in experiments~\cite{SongETAL:05}. Second, in the standard random graph model, assemblies are shown to correspond to the firing of $k$ neurons, with most of them in a fixed set of size $(1+o(1))k$ with at most $o(k)$ outside this set. On the other hand, what has been observed is that assemblies represent increased firing activity of a relatively small but significantly larger than $k$ subset of neurons for a period of time~\cite{durstewitz2000neurocomputational,buzsaki2019brain}. An exciting aspect of our investigation is a rigorous explanation of this phenomenon.

\subsection{Main Results}

As a warm-up, we consider the infinite limit, i.e., the {\em continuous} interval $[0,1]$ in one dimension. Then we turn to the discrete setting of graphs, with vertices chosen uniformly from the $d$-dimensional unit cube. While the brain motivation applies directly to $d=2,3$, higher dimension is also relevant and interesting, as vertex location could indicate some set of relevant features (e.g., type of neuron). 

\paragraph{A Continuous Process.}
A natural abstraction of the $k$-cap process on geometric random graphs is to consider what happens when the number of vertices, $n$, goes to infinity. On a finite graph, the input to a discrete vertex $v$ is the sum of its edges from $A_t\subset V=\{1/n, 2/n, \dots, 1\}$. In the infinite limit, we assume that $A_t\subset[0, 1]$ is a set of measure $\alpha$, leading to a corresponding $\alpha$-cap process. The input to a given point $v\in [0,1]$ is the integral of the edge probability function over $A_t$ and the $\alpha$ fraction of points with the highest input will form $A_{t+1}$. We provide a formal definition of this process in Section \ref{sec:cts}.

	This continuous abstraction leads to a clean convergence phenomenon. We find that $A_t$ converges to a single interval of length $\alpha$. The number of steps to convergence depends on how large the derivative of the edge probability function can be. The exact result proven in this paper is stated below.
	
\begin{restatable}{theorem}{thmcts}
	\label{thm:cts_convergence}
	Let $A_0$ be a countable set of intervals in [0,1] and $g$ be the edge probability function. For any differentiable, even, nonnegative and integrable function $g:[0,1]\rightarrow \mathbb{R}_+$ with $g'(x) < 0$ for all $x > 0$, the $\alpha$-cap process converges to a single interval of width $\alpha$. Moreover, the number of steps to convergence is 
	\[
	O\left(\frac{\max_{[0,1]}|g'(x)|}{\min_{[\frac{\alpha}{8},1]}|g'(x)|}\right).
	\]
\end{restatable}

Note that the above conditions capture any distance function that decays smoothly with the distance between its endpoints, e.g., the well-known Gaussian kernel.
This process is deterministic given the initial choice of $A_0$. We are able to bound the convergence using a simple potential function: 

\begin{center}
    {\em The distance between the medians of the leftmost and rightmost intervals of $A_t$ decreases.} 
\end{center}

In this way, the intervals are ``squeezed" together until they collapse into one. In this continuous version, any sub-interval in $[0, 1]$ of length $\alpha$ is a fixed point; if $A_t = [a, b]$, then $A_{t+1} = A_t$. Moreover, a single interval is the only possible fixed point.

\paragraph{Formal Definition of the Discrete Process. }
Now we turn to the main setting of this paper, the $k$-cap process on a finite directed graph $G$. The following symbols will be used for the rest of the paper. Let $n$ be the number of vertices in the graph and $k$ be the number of vertices activated at each step. $A_t$ represents the set of $k$ vertices activated at step $t$ for $t=0,1,2,\ldots$. Let $\mathbbm{1}_{(x,y)}$ be the indicator variable for the directed edge between two vertices $x$ and $y$.
\begin{definition}[$k$-cap Process]
	\label{k-cap-definition}
	Assume $A_t \subset \{1, 2, \dots, n\}$, and $|A_t| = k$.
	Let $F_t:\{1, 2, \dots, n\} \rightarrow \{0, 1, 2, \dots,  k\}$ be the synaptic input function at time $t$, defined as follows:
	$$F_t(x) = \sum_{y\in A_t} \mathbbm{1}_{(y, x)}$$
	Let $C_t$ be the smallest integer such that $|\{x \mid F_t(x)> C_t\}|\leq k$,   and let 
	$$A_{t+1} = \{x \mid F_t(x)> C_t\} \cup A_{t+1}^*$$
	where $A_{t+1}^*$ is a set of points sampled at random from $\{x \mid F_t(x) = C_t\}$ such that $|A_{t+1}| = k$. 
\end{definition}
In the $k$-cap process, $A_{t+1}$ is chosen as the $k$ vertices with the highest degree from $A_t$. If there are ties, the remaining vertices are chosen uniformly at random from the set of vertices with the next highest degree. $A_0$ can be instantiated in any way, but we assume that it is chosen uniformly at random from the set of vertices. 

We analyze the convergence of the $k$-cap process on a $d$-dimensional Gaussian geometric random graph; the probability of an edge between two vertices with hidden variables $x$ and $y$ is a Gaussian kernel; i.e., $\mathbb{P}((x, y) \in G) = g(x, y) = \exp\left(-(x-y)^2/(2\sigma^2)\right)$.
Here, $\sigma$ is a parameter of the model. For simplicity, we use $x$ to represent both the vertex and its hidden variable in $[0, 1]^d$. Throughout this paper, we will use the terms \textit{point} and \textit{vertex} interchangeably. 

\begin{definition}[$d$-dim Gaussian Geometric Random Graph]
	\label{graph-definition}
 Let $G=G_\sigma = (V,E)$. Let $V=\{v_1, v_2, \dots, v_n\}$ where each vertex is a point chosen uniformly at random in $[0,1]^d$. Each directed edge $(x,y)$ is present in the graph with probability $$\mathbb{P}(\mathbbm{1}_{(x,y)}) = g(x,y) = \exp\left(-\frac{\|x-y\|_2^2}{2\sigma^2}\right)$$
\end{definition}
\noindent Unless otherwise stated, assume $\|x-y\| = \|x-y\|_2$ is the Euclidean distance.
\paragraph{Convergence of the Discrete Process.} 

The discrete process on graphs turns out to exhibit much more complex behavior than the continuous variant. With the randomness induced by the choice of edges, the convergence behavior also becomes probabilistic rather than ending in a fixed set or distribution. We will prove that in the interesting range of $\sigma$, the cap $A_t$ converges with high probability to lie within a small ball (an interval when $d=1$). Note that the process will not converge to a fixed set of points; it will instead randomly oscillate within a small subset of the hidden variable space, corresponding to a small dense subgraph of $G$. The reasons for this type of convergence are discussed in Section \ref{sec:proof-outline}. We believe that this behavior is both interesting mathematically and of relevance to modeling the brain; it provides a new perspective on the notion of an assembly.

Given that $A_t$ will not converge to a fixed set of points, we will determine the size and structure of the set of likely points. First, it is not immediately obvious that the set will converge to having its support in a single small interval, e.g., it may be the case that the set of points which fire are split between two weakly connected subgraphs. Additionally, we are interested in the \textit{width} of the set of likely points. When two points are at a distance $\Theta(\sigma)$, the probability that they are connected by an edge is a constant, bounded away from $0$. At a larger distance, the probability drops off quickly; therefore, we might expect that $A_t$ will be contained within an interval of size $O(\sigma)$. In fact, as stated in Theorem \ref{thm-main-at-multidim} below, we find that it will converge to an interval much smaller than $\sigma$.

\begin{theorem}
\label{thm-main-at-multidim}
There exists a $t^* \le \ln^c k$, for a constant $c$, such that $A_t$ can be covered by a single ball of radius $\Theta\left(\sigma \sqrt{\ln k/k}\right)$.
\end{theorem}

\paragraph{Evolution of the $k$-cap.}
    \begin{wrapfigure}{h!}{0.5\textwidth}
    \includegraphics[width = 0.5\textwidth]{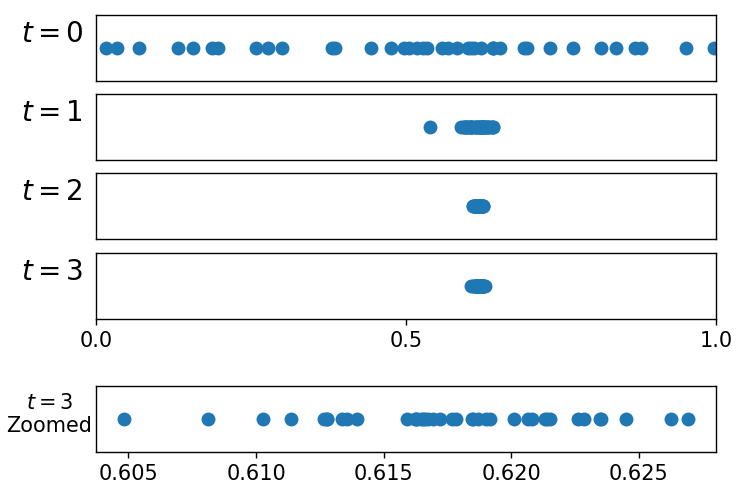}
	\caption{The $k$-cap at 4 time steps, with parameters $n=90000$ and $k=40$. }
	\label{fig:simulation2}
\end{wrapfigure}
  The evolution of the structure of $A_t$ also reveals interesting properties of the $k$-cap process on random graphs. 
	In the first step, $A_0$ is uniformly distributed. However, due to the the {\em Poisson clumping} phenomenon~\cite{aldous2013probability},  there will be a several regions of the ball with a higher concentration than average. As we will show, for $\sigma$ sufficiently small, $A_1$ will be concentrated within $k^{1/4+o(1)}$ balls which are small compared to $[0, 1]^d$ (This result is described formally in Theorem \ref{thm:A1}).
	As $t$ increases, all but one of these balls will diminish and disappear. With high probability, each ball will shrink by a fixed fraction at each step. After this, we will show that one ball will ``win" over the others.
 
	When $n$ is sufficiently large, $A_t$ eventually lies within a subset of a ball of radius $\sigma\sqrt{\ln k/k}$. The distribution of $A_t$ is fairly uniform in the core of the ball, and the probability drops toward $\frac{1}{n}$ as it moves to the edges. Figure~\ref{fig:simulation2} graphs the points which fired in a simulation of the $k$-cap process over 4 time steps. In both, the graph is random with each edge added with probability $\exp(-|x-y|^2 k^2/2)$. 

\paragraph{Parameter Range. } In this paper, we focus on $\sigma = \Theta(1/k^{1/d})$. The justification for this parameter range lies in the concentration behavior of uniform random variables. The soft geometric random graph model can be thought of as an approximation of an interval graph with radius $\Theta(\sigma)$. Let $U=\{U_1, U_2, \dots U_k\}$ be a set of $k$ random variables, each chosen uniformly at random in $[0, 1]^d$. For a given radius $r$, we can compute the maximum number of points which are likely to fall into a ball of radius $r$ (this result is described in Lemma \ref{lem-max-degree}).  Note the expected number of points in a ball $I$ is $k\Vol(I)$, and $\Vol(I) = \Theta(r^d)$; as $r$ increases, the maximum degree approaches the expected value. This phenomenon means that the concentration behavior starts to disappear as $\sigma$ increases past $1/k^{1/d}$. On the other hand, as $r$ decreases, the maximum degree approaches $1$. Therefore, we focus on an intermediate range of $\sigma$ where the concentration of $A_0$ leads to interesting behavior. 

\paragraph{Organization.} 
We next discuss the discrete process in detail and present our precise findings about its convergence. Following that, in Section~\ref{sec:discrete-proof}, we prove our main theorem about this, starting with a detailed exposition of related probabilistic considerations. In Section~\ref{sec:cts} we prove the convergence of the continuous process. 

\subsection{Analysis Outline}

\label{sec:proof-outline}

An important quantity in the analysis will be the probability that $F_t(x)$ exceeds a given threshold: $p_{C, t}(x) = \mathbb{P}(F_t(x) \geq C)$.
The probability is conditioned on the random choice of edges in the graph. If $n$ is sufficiently large, we can make the assumption that with high probability, the random variable $F_t(x)$ depends only on the edges from $x$ to $A_t$ and is independent of $A_0, \dots, A_{t-1}$. This will be proven formally when it becomes relevant. When not otherwise specified, $p_t(x) = p_{C_t, t}(x)$ is the probability that $F_t(x)$ strictly exceeds the $k$-cap threshold $C_t$, conditioned on the choice of the set $A_t$.  Figure~\ref{fig:simulations} shows an empirical demonstration of how the functions $\mathbb{E}F_t(x)$ and $p_t(x)$ evolve over time. 

\begin{figure}[h]
	\centering
    \centering
	\subfloat[\centering $t=0$]{{\includegraphics[width=0.45\textwidth]{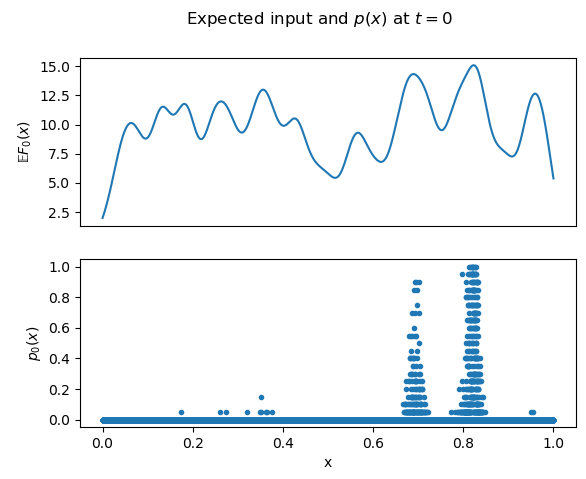} }}
	\qquad
	\subfloat[\centering $t=2$]{{\includegraphics[width=0.45\textwidth]{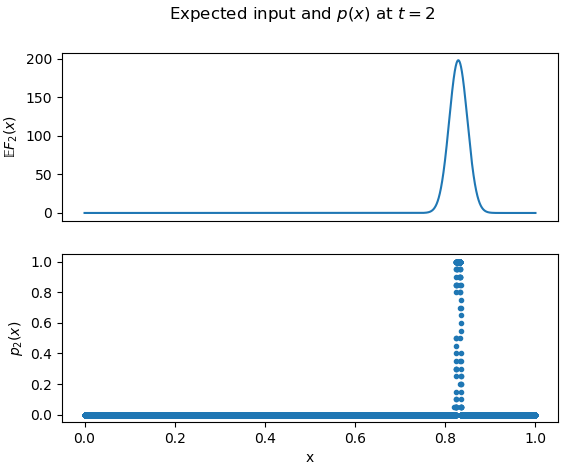} }}
	\caption{Expected input and $p_t(x)$ at two different time steps for $n=20000$ and $k=200$. The probability $p_t(x)$ was estimated by fixing the firing set, and then repeatedly redrawing the graph edges. The probability of $x$ is defined as the percentage of times $x$ was chosen by the top-k function. The left figure illustrates the input due to $A_0$, which is uniformly random on $[0,1]$. The right figure shows the input two steps later, when the set $A_2$ has converged to lie in a single, narrow interval. }
	\label{fig:simulations}
\end{figure}

The proof will be divided into three parts. First, we will characterize the structure of $A_1$.
The set $A_0$ is drawn uniformly at random from the set of vertices. However, due to the {\em Poisson clumping} phenomenon~\cite{aldous2013probability}, there will be a few regions of the hidden variable space which have a significantly higher concentration of points than average. Vertices with hidden variables in these regions will have larger inputs $F_0$. We have the following theorem describing points which have a significant probability of exceeding the threshold at Step $0$. 

\begin{restatable}{theorem}{theoremAone}
	\label{thm:A1}
	Let $n = k^{\beta}$ for some constant $\beta \geq 2+d$. 
	Then, with high probability, $A_1$ can be covered by $k^{\frac{1}{4}+o(1)}$ balls, each of radius $O(\sigma\sqrt{\ln \ln k})$ and pairwise separated by a distance of at least $2\sigma \sqrt{\ln n}$.
\end{restatable}
By the assumption that $\sigma = \Theta(k^{-1/d})$, the set $A_1$ is contained in a small region relative to $[0, 1]^d$. 

The next step of the analysis will show that $A_t$ will gradually become concentrated in a single ball of radius $O(\sigma\sqrt{\ln k/k})$. The key idea is that each individual ball shrinks with high probability at each step, almost in place. 

The statement of this lemma is below. We use $r(S)$ to denote the radius of the smallest ball containing a set $S$.

\begin{restatable}{lem}{lemmainAt}
	\label{lem:main_At-multidim}
	Suppose that $A_t \subset I_1 \cup \dots\cup I_i$, where $i = O(k^{1/4+o(1)})$,  and for each $j$, $I_j$ is a ball of radius $r_j$, bounded by $r(I_j) = \Omega(\sigma\sqrt{\ln k/|I_j \cap A_t|})$ and $r(I_j)<C\sigma \sqrt{\ln \ln k}$ for some constant $C$. Also assume that the distance between any two balls is at least $2(1-o(1))\sigma\sqrt{ \ln n}$. 
At the next step, with high probability, 
	$A_{t+1} \subset I_1' \cup I_2'\cup \dots \cup I_i'$,
	where $d(I_j, I_j') = \max_{x \in I_j'}\min_{y \in I_j}\lVert x-y\rVert < 5(r(I_j) - r(I_j'))$, 
	and $r(I_j')\le (1- \frac{1}{(\ln k)^c}) r(I_j)$ for an absolute constant $c$.
\end{restatable}
Up to some small deviations, each ball determined by time Step $1$ will either disappear or shrink slightly in place. Moreover, the surviving balls will remain separated; the maximum distance moved by a single ball by time $t$ is $5 (r(I^{(0)}_j) - r(I^{(t)}_j)) = O(\sigma \sqrt{\ln \ln k})$. Using this, we will show convergence to a structure contained within a single ball of radius $O(\sigma \sqrt{\ln k/k})$. This is the content of Theorem~\ref{thm-main-at-multidim}.

This structure is, to within a log factor, the smallest subgraph we can expect $A_t$ to converge to. Within a region of radius $O(\sigma k^{-1/2})$, the edge probability is greater than $e^{-1/k}$; hence, the degree of any vertex to $A_t$ will be almost constant in this interval. The probability that a vertex fires remains uniform in the center and drops off toward $\frac{1}{n}$ at the ends of the interval.

Finally, we show that conditioned on the structure of the graph, almost all of $A_t$ is contained within a ball of radius $O(\sigma k^{-1/3+\epsilon})$ for all $t\geq t^*$.
\begin{theorem}
\label{thm-all-t-multidim}
For all $t \geq t^*$, with high probability, there exists a ball $I_t$ with radius $r = \sigma k^{-1/3+\epsilon}$, for a constant $\epsilon > 0$, such that $|A_t \cap I_t| > k- k^{2/3}$.
\end{theorem}
The proof of this theorem directly implies the following structural property of geometric random graphs, which holds for all $S \subset V$ of size $k$ that are mostly contained in a small ball in $[0,1]^d$. With high probability over all such sets, the set of $k$ vertices with the highest degree from $S$ are also mostly contained in a small ball.
\begin{corollary}
Let $G=(V,E)$ be a geometric random graph such that for every vertex $x \in V$, its location $h_x$ is chosen uniformly at random from $[0,1]^d$ and for every pair $x,y \in V$, $\mathbb{P}((x,y) \in E) = e^{-\lVert h_x - h_y\rVert k^{2/d}}$, where $k = O(|V|^{\frac{1}{2+d}})$. Let $r = k^{-1/d - 1/3 + \epsilon}$ for any $\epsilon >0$.
Then, with high probability (over the edges of $G$), for every set $S \subset V$ of size $k$, if at least $k-k^{2/3}$ points of $S$ are contained in a ball of radius $r$, then there exists a ball of radius $r$ which contains at least $k-k^{2/3}$ points of $S'$, the set of $k$ points with the highest degree from $S$. 
\end{corollary}
\section{Analysis of the discrete $k$-cap process}\label{sec:discrete-proof}
In Appendix \ref{sec:prob-prelims}, we will introduce a few general results on probability which play key roles in the proof. In Section \ref{sec:a1}, we describe the structure of $A_1$ given the random initial firing set $A_0$. Lastly, in Section \ref{sec:at}, we investigate the evolution of $A_t$ as $t$ increases, and we prove the main theorem on the convergence of the process, stated in Theorem \ref{thm-main-at-multidim}.

\subsection{Characterization of $A_1$}
Since the initial firing set $A_0$ is chosen uniformly at random from $V$, $A_1$ is concentrated near dense sections of $A_0$. We argue that the probability that $x\in A_1$ can be characterized by conditioning on the number of points of $A_0$ within $\tilde{O}(\sigma)$ of $x$. By analyzing the distribution of dense subsets of a set of uniform random variables, we show that $A_1$ is contained within a union of $k^{1/4+o(1)}$ small balls. 

First, in Lemma \ref{lem:c1bound}, we give a lower bound on the first threshold, $C_0$, by examining the maximum number of uniform random points within a ball of radius $r$. There are, with high probability, at least $k$ vertices of $V$ which connect to the every point in the intersection of the ball with $A_0$. 

Next, Lemma \ref{lem-helper-a1} shows that if $|\{y \in A_0 : \lVert x-y\rVert = O(\sigma\sqrt{\ln \ln k})\}|$ is not large, $x$ has a very small probability of achieving an input of $C_0$. This implies that $A_1$ must be solely contained within high-density regions of $A_0$.

Finally, we combine these two lemmas to prove Theorem \ref{thm:A1}, restated later in this section. Since the number of high-density regions can be bounded of $A_0$ using a combinatorial argument, $A_1$ must be contained within $k^{1/4 + o(1)}$ small balls. 
\label{sec:a1}
\begin{lemma}
		\label{lem:c1bound}
		With probability $1-o(1)$ (where $\ln^{(3)} k = \ln \ln \ln k$):
		$$C_0 \geq \frac{\ln k}{\ln \ln k}\left(1+\frac{1}{4} \frac{\ln^{(3)} k}{\ln \ln k}\right)$$
	\end{lemma}
	\begin{proof}
		Consider a graph constructed on $A_0$ as follows. For any $a_1, a_2 \in A_0$, add an edge if $\lVert a_1 - a_2 \rVert < \frac{1}{2}\sigma \sqrt{\ln \ln k}$. Denote the maximum degree of this graph as $\Delta_k$. $\Delta_k + 1$ is the maximum intersection of $A_0$ with a circle of radius $r=\frac{1}{2}\sigma \sqrt{\ln \ln k}$. 
  
  By Lemma \ref{lem-max-degree}, 
  \begin{align*}
  \Delta_k &\geq \frac{\ln k}{\ln \ln k - \ln(\sigma^d k(1/4\ln \ln k)^{d/2})} \\
  &=\frac{\ln k}{\ln \ln k - d/2\ln^{(3)}k+O(1)}\\
  &\geq \frac{\ln k}{\ln \ln k} \left( 1 + \eta\right) = M_k \qquad \mbox{where } \eta = \frac{\ln^{(3)} k}{4\ln \ln k}. 
  \end{align*}

 As per above, there exists a ball with $M_k$ points of $A_0$ almost surely. Call this ball $I$.

		The maximum distance within the ball is $\sigma\sqrt{\ln \ln k}$. For any $x \in I$, the probability that $x$ connects to $M_k$ points is at least $g(x, x+v)^{M_k}$, where $v$ is a vector of size $\sigma\sqrt{\ln \ln k}$. Substituting:
		\begin{align*}
		    \mathbb{P}(F_0(x) \geq M_k) &\geq \exp\left(-\frac{(\sigma \sqrt{\ln \ln k})^2}{2\sigma^2}\right)^{M_k} =\exp\left(-\frac{\ln \ln k}{2}\right)^{(1+\eta)\frac{\ln k}{\ln \ln k}}\\&= k^{-\frac{(1+\eta)}{2}}
		\end{align*}
		By Lemma \ref{lem-n-uniform-random},  $|I\cap V| = \Omega\left(\frac{\Vol(I) \cdot n}{\log n}\right) = \Omega\left(\sigma^d (\ln \ln k)^{d/2}\cdot \frac{n}{\log n}\right)$. By the assumption that $\sigma = \Theta(k^{-1/d})$ and $n=k^{\beta}$, there are $k^{\beta-1 - o(1)}$ vertices of $G$ in $I$. 
		
		The expected number of points with degree $M_k$ from $|A_0 \cap I|$ is at least $k^{\beta-1 - o(1)}k^{- \frac{(1+\eta)}{2}}$, which is much greater than $k$. Therefore, there are at least $k$ points with input $M_k$ with high probability. This implies that the threshold $C_0$ is bounded from below by $M_k$. 
	\end{proof}

\begin{lemma}
	\label{lem-helper-a1}
		For any $x \in [0,1]^d$, define $B_r(x) = \{y\in[0,1]^d : \lVert x-y\rVert < r\}$.
		Let $r=\sigma \sqrt{24\beta\ln \ln k)}$. Suppose that the overlap between $B_r(x)$ and  $A_0$ is at most $\frac{3\ln k}{4\ln \ln k}$. Conditioned on this event, the probability that $x\in A_1$ is at most $\frac{1}{n^3}$. 
	\end{lemma}
	\begin{proof}

		Let $r = \alpha\sigma \sqrt{\ln \ln k}$. Suppose $|B_r(x) \cap A_0| \leq\frac{3\ln k}{4\ln \ln k}$. 
		
		By Lemma \ref{lem:c1bound}, if $x \in A_1$, then $F_0(x) \geq C_0 \geq \frac{\ln k}{\ln \ln k}$. Hence, by assumption, $x \in A_1$ only if it achieves an input of $M=\frac{\ln k }{4\ln \ln k}$ from outside $B_r(x)$. 
		
		Conditioned on $|B_r(x)\cap A_0| = \lambda k$, the remaining $(1-\lambda)k$ points are distributed uniformly on $[0,1]^d\setminus B_r(x)$. 
		
		By definition, $\mathbb{P}(\mathbbm{1}_{(x, Y)}\mid Y=y) = g(x,y)=\exp\left(-\lVert x-y\rVert^2/2\sigma^2\right)$. Let $f_r$ be the conditional distribution function of $\lVert y-x\rVert$, which has support on $(r, \sqrt{d}]$ ($\sqrt{d}$ being the longest diagonal of the hypercube $[0,1]^d$). Define $\partial B_r(x) = \{y\in[0,1]^d : \lVert x-y\rVert = r\}$ to be the spherical shell of radius $r$ around $x$. 
         
  $$f_r(\rho) = \frac{\Vol(\partial B_\rho(x) \cap [0,1]^d)}{1-\Vol(B_r(x))}$$

        The boundaries of the hypercube make $f_r$ somewhat difficult to calculate. Therefore, we will ignore the boundaries and set $f_r(\rho) < \frac{\Vol(\partial B_\rho(x))}{1-\Vol(B_r(x))}$. Note that $\Vol(\partial B_\rho(x)) = \frac{2\pi^{d/2}}{\Gamma(d/2)}\rho^{d-1}$.

		\begin{align*}
		\mathbb{E} [\mathbbm{1}_{(x,y)} \mid y\notin B_r(x)] &= \int_r^{\sqrt{d}} e^{-\rho^2/(2\sigma^2)} f_r(\rho) \, d\rho\\
        &\leq \frac{1}{1-\Vol(B_r(x))}\int_r^{\sqrt{d}} e^{-\rho^2/(2\sigma^2)}\Vol(\partial B_\rho(x))\, d\rho\\
        &\leq 2\frac{2\pi^{d/2}}{\Gamma(d/2)}\int_r^{\infty} \rho^{d-1} e^{-\rho^2/(2\sigma^2)}\, d\rho\\
        &= \frac{4\pi^{d/2}\sigma^{d}}{\Gamma(d/2)}\int_{r/\sigma}^{\infty} z^{d-1}e^{-z^2/2}\, dz\\
		\end{align*}

        This integral can be estimated by observing that for $r$ sufficiently large, $z^{d-1} e^{-z^2/2}$ is decreasing on $[r/\sigma, \infty]$. Therefore, 
        \begin{equation}
        \label{eq-gauss-tail}
        \int_{r/\sigma}^{\infty} x^{d-1}e^{-x^2/2}\, dx\leq \left(\frac{\sigma }{r}\right)^{d-2}\int_{r/\sigma}^{\infty} xe^{-x^2/2}\, dz=\left(\frac{r}{\sigma}\right)^{d-2}e^{-r^2/(2\sigma^2)}
        \end{equation}
        Returning to the original equation, for any $d\geq 1$:
        $$ \mathbb{E} [\mathbbm{1}_{(x,y)} \mid y\notin B_r(x)] \leq\frac{4\pi^{d/2}\sigma^{d}}{\Gamma(d/2)}\left(\frac{r}{\sigma}\right)^{d-2}e^{-r^2/(2\sigma^2)}  = \Theta(1)\sigma^{d}\left(\frac{r}{\sigma}\right)^{d-2}e^{-r^2/2\sigma^2}$$

		Substituting $r=\sigma \alpha \sqrt{\ln \ln k}$, this equals $\Theta(1)\sigma^d (\ln k)^{-\alpha^2/2} ( \sqrt{\ln \ln k})^{d-2}$. Again recalling $\sigma^d = \Theta(1/k)$, 
  $$\mathbb{E} [\mathbbm{1}_{(x,y)} \mid y\notin B_r(x)] \leq p = O(1/k)(\ln k)^{-\alpha^2/2} (\alpha \sqrt{\ln \ln k})^{d-2}$$.
     We can bound the distribution of $F_0(x)$ by a binomial with probability $p$. In particular, $\mathbb{P}(F_0(x) > C_0)$ is bounded above by $\mathbb{P}\left(B > \frac{\ln k}{4 \ln \ln k}\right)$, where $B \sim Bin(k,p)$. This quantity can be tightly bounded using Lemma~\ref{lem-bin-tight}. 
            $$\mathbb{P}\left(B > \frac{\ln k}{4 \ln \ln k}\right) \leq \exp\left( - k D\left(\frac{\ln k}{4k\ln \ln k} \mid \mid p\right)\right)$$
        Bounding the divergence term:
        \begin{align*}
            D\left(\frac{\ln k}{4k\ln \ln k} \mid \mid p\right) &= \frac{\ln k}{4k\ln \ln k}\ln \frac{\ln k}{4kp\ln \ln k} + (1-\frac{\ln k}{4k\ln \ln k}) \ln \frac{1-\frac{\ln k}{4k\ln \ln k}}{1-p}\\
            &\geq \frac{\ln k}{4k\ln \ln k}\ln \frac{\ln k}{4kp\ln \ln k} + p-\frac{\ln k}{4k\ln \ln k} \text{ using }\ln x \geq 1-1/x \;\forall x>0\\
            &=\frac{\ln k}{4k\ln \ln k}\ln \frac{\ln k^{1+\alpha^2/2}}{O(1)(\ln \ln k)^{1 + (d-2)/2}} + \Theta(\frac{1}{k})(\ln k)^{-\alpha^2/2} (\alpha \sqrt{\ln \ln k})^{d-2}-\frac{\ln k}{4k\ln \ln k}\\
            &\geq\frac{\ln k}{k}\left[\frac{1+\alpha^2/2}{4} - \frac{1+(d-2)/2}{4}\frac{\ln^{(3)}k}{\ln \ln k} - \frac{O(1)}{\ln \ln k}\right]
        \end{align*}
		
		Plugging this into the original bound, $\mathbb{P}\left(B > \frac{\ln k}{4 \ln \ln k}\right) \leq k^{-\frac{1+\alpha^2/2}{4} + o(1)}$

		Since $n = k^{\beta}$ by definition, we can choose $\alpha = \sqrt{24\beta}$. Then, $\mathbb{P}(F_0(x) > C_0) \leq \mathbb{P}\left(B > \frac{\ln k}{4 \ln \ln k}\right)\leq 1/n^3$
	\end{proof}
\theoremAone*

\begin{proof}[Proof of Theorem \ref{thm:A1}]

        We apply Lemma \ref{lem-helper-a1} and take the union bound over all $x$ in the graph to conclude the following: with probability, $1-1/n^2$, $x\in A_1$ only if the ball $B_r(x)$, where $r=\sigma\sqrt{24\beta\ln \ln k}$,   contains more than $\frac{3\ln k}{4\ln \ln k}$ points. Since $1/n^2$ is summable, this is true almost surely. 
        
        Since the expected number of points of $A_0$ in $B_r(x)$ is $k * \Vol(B_r(x)) = O((\ln \ln k)^{d/2})$, the number of such high density regions will be relatively small. 
        
        To argue this, we consider $d+1$ overlapping partitions of $[0, 1]^d$ into boxes. Let $L = 2\sigma\sqrt{24\beta\ln \ln k}$. First, tile $[0,1]^d$ with boxes of width $L$. Then, shift each interval by half its width in each dimension, leading to $d$ alternate partitions of $[0,1]^d$.
        
        For any $x$, the ball $B_{\sigma \sqrt{24\beta\ln \ln k}}(x)$ must be fully contained in a box $I_i$ in at least one partition for some index $i$. Consider the probability that a given box $I_i$ contains $\frac{3\ln k}{4\ln \ln k}$ points of $A_0$. 
        
        For each point in $A_0$, the probability that it lands in $I_i$ is $\Vol(I_i) = (2\sigma\sqrt{24\beta \ln \ln k})^d = \Theta(1)(\ln \ln k)^{d/2}/k$. Therefore, the number of points in $I_i$ is $|A_0 \cap I_i| \sim Bin(k, \Vol(I_i))$. 
        
        Using the binomial bound in Lemma \ref{lem-bin-tight}, the probability that $|I_i \cap A_0|$ exceeds $3\ln k/4\ln \ln k$ is at most:
        \begin{equation}
        \label{eq-interval-bound}
            \mathbb{P}\left(|I_i \cap A_0|>\frac{3\ln k}{4\ln \ln k}\right) \leq \exp\left(-kD\left(\frac{3\ln k}{4k \ln \ln k} \mid \mid \Vol(I_i)\right)\right)
        \end{equation}
Bounding the divergence term (Using the inequality $\ln x \geq 1-1/x$):
        \begin{align*}
            D&\left(\frac{3\ln k}{4k \ln \ln k} \mid \mid \Vol(I_i)\right)=\frac{3\ln k}{4k \ln \ln k} \ln \frac{\Theta(1)\ln k}{(\ln \ln k)^{1+d/2}} + \left(1-\frac{3\ln k}{4k \ln \ln k}\right) \ln \frac{1-\frac{3\ln k}{4k \ln \ln k}}{1-\Theta(1)(\ln \ln k)^{d/2}/k}\\
            &\geq\frac{3\ln k}{4k} - \frac{O(1)\ln k\ln^{(3)}k}{k\ln \ln k}  + \left(1-\frac{3\ln k}{4k \ln \ln k}\right)\left(1-\frac{1-\Theta(1)(\ln \ln k)^{d/2}/k}{1-\frac{3\ln k}{4k \ln \ln k}}\right)\\
            &\geq \frac{3\ln k}{4k} - \frac{O(1)\ln k\ln^{(3)}k }{k\ln \ln k}  -\frac{3\ln k}{4k \ln \ln k}+\Theta(1/k)(\ln \ln k)^{d/2}\\
            &=\frac{1}{k}\left[\frac{3}{4}\ln k -\frac{O(1)\ln k\ln^{(3)}k}{\ln \ln k}\right]\\
        \end{align*}
        
        Substituting back into equation \ref{eq-interval-bound}:
        $$\mathbb{P}\left(|I_i \cap A_0|>\ln k/\ln \ln k\right) \leq \exp\left(-\frac{3}{4}\ln k +\frac{O(1)\ln k\ln^{(3)}k}{\ln \ln k}\right)=k^{-\frac{3}{4}+\frac{O(1)\ln^{(3)}k}{\ln \ln k}}$$
        
        There are $\frac{d+1}{\Vol(I_i)} = \Theta\left(\frac{k}{(\ln \ln k)^{d/2}}\right)$ such intervals. The size of $|B_i \cap A_0|$ for each partition can be thought of as the loads in a `balls into bins' problem; thus, the number of points in non-overlapping boxes are negatively correlated.
        
        With high probability, the number of such intervals with enough points is $k^{1/4 + o(1)}$. 
        
        The loads of the bins $I_i$ are invariant to permutation; therefore, the probability that two intervals within a distance of $2\sigma \sqrt{\ln n}$ have a large enough load is $o(1)$.
        
        The same can be said for each shifted partition. Therefore, the bins which achieve high input are of size $\Theta(L) = \Theta(\sigma\sqrt{\ln \ln k})$ and separated by a distance of $2\sigma \sqrt{\ln n}$. 
\end{proof}
\subsection{Convergence of $A_t$}
	\label{sec:at}
	In this section, we will prove the main Lemma~\ref{lem:main_At-multidim}, which will lead to the proof of Theorem~\ref{thm-main-at-multidim}.
    In Theorem~\ref{thm:A1}, we have proved that $A_1$ can be covered by $k^{1/4+o(1)}$ balls of radius $O(\sigma \sqrt{\ln \ln k})$, and separated by at least $2\sigma \sqrt{\ln n}$. There are two key properties of this system which make the analysis tractable. First, the separation condition allows us to analyze each interval as a separate system. If $x \in I_a$ and $y \in I_b$, $g(x, y) < \exp\left(-\frac{4\sigma^2\ln n}{2\sigma^2}\right) = n^{-2(1-o(1))}$. Therefore, with high probability, the subgraphs defined by $I_a$ and $I_b$ are independent; this means that all $x\in A_t$ will not receive input from outside its interval. Second, since the graph is directed, the edge $\mathbbm{1}_{(x,y)}$ is independent of $\mathbbm{1}_{(y,x)}$. Additionally we prove in Lemma \ref{lem:independence} that for any $t = \polylog(k)$, all points which fire at $t$ are `new' (i.e., they have not fired at a previous step) with high probability. This lets us make the simplifying assumption that $F_t(x)$ is a sum of independent indicators. Using these two key simplifications, we prove that with high probability, each separated interval shrinks to a size of $O(\sigma\sqrt{\ln k/k})$. 

We will suppose that the hypothesis of Theorem~\ref{thm:A1} holds for a step $t\geq1$; $A_t$ can be covered by $O(k^{1/4+o(1)})$ sufficiently separated balls. Then, we will prove that the separation and coverage continue to hold by induction. 
    
    Define $A_t \subset I_1 \cup I_2 \cup \dots \cup I_i$, where each $I_j$ is a ball of radius $O(\sigma \sqrt{\ln \ln k})$, and all pairs $I_a$, $I_b$ are separated by a gap of at least $2(1-o(1))\sigma \sqrt{\ln n}$.  Also define $E[x] = \mathbb{E}F_t(x) = \sum_{z \in A_t}g(x, z)$, and $V[x]^2 = \Var F_t(x) = \sum_{z \in A_t}g(x, z)(1-g(x, z))$. Note that $E[x]$ and $V[x]$ depend implicitly on $t$.

    The following lemmas will be used to bound $C_t$ at each time step. Using this, we can get precise bounds on $\mathbb{P}(F_t(x) > C_t)$.
	\begin{lemma} 
		\label{lem-Ex}
  
	 For any vector direction v and point $x\in [0,1]^d$, $\left|\nabla_v E[x]\right|< \frac{k}{\sigma}\sqrt{d/e}$. 
	\end{lemma}
\begin{proof}[Proof of Lemma~\ref{lem-Ex}]
	
    For any $i\in\{0,1,\dots,d-1\}$:
	$$\frac{\partial}{\partial x_i} E[x]= \sum_{z \in A_t}\frac{\partial}{\partial x_i}g(x, z) = \sum_{z \in A_t} -\frac{x_i-z_i}{\sigma^2} \exp\left(\frac{-\lVert  x-z \rVert^2}{2\sigma^2}\right)$$
	Let $\hat{x} = \frac{x-z}{\sigma}$. The maximum of $\left|\hat{x} \exp\left(\frac{-\hat{x}^2}{2}\right)\right|$ occurs at $e^{-1/2}$. Thus, we have 
	$$ \left|\frac{\partial}{\partial x_i}  E[x]\right| \leq \sum_{z \in A_t}\frac{|x_i-z_i|}{\sigma^2} \exp\left(\frac{-\lVert x-z\rVert^2}{2\sigma^2}\right)<\sum_{z \in A_t}\frac{1}{\sigma} e^{-1/2} =  \frac{k}{\sigma}e^{-1/2}$$
    For any unit vector $v$:
    $$\left|\nabla_v  E[x]\right| < \frac{k}{\sigma} e^{-1/2} v \cdot \mathbf{1} = \sqrt{d} e^{-1/2} \frac{k}{\sigma}$$
\end{proof}
	For this proof to be viable, we will need to show that $F_t(x)$ is the sum of independent indicators. By definition of the  graph structure, each edge $\mathbbm{1}_{(y, x)}$ is drawn independently. However, we will also need to show that, for each $y \in A_t$, its edges $\mathbbm{1}_{(y, x)}$ have not been used in previous computations. This follows from the next lemma.
	\begin{lemma} 
		\label{lem:independence}
		Suppose $t = O((\ln k)^c)$ for a constant $c$. Then, with probability at least $1-\frac{1}{k^{1/2-o(1)}}$ 
		$$A_0 \cap A_1 \cap \dots \cap A_t =\emptyset$$ 
	\end{lemma}
\begin{proof}
		Suppose at time $s$, $\{A_0, A_1 \dots A_s\}$ are pairwise disjoint. Therefore, at time $s$, the edges $\{\mathbbm{1}_{(y,x)}: y \in A_s, x\in [n]\}$ have not been examined by the $k$-cap function, and they are conditionally independent. 
  
		Now, we will compute the probability that $|A_{s+1} \cap A_i| > 0$ for some $i\leq s$. By Lemma~\ref{lem-n-uniform-random}, there are at least $\Theta(1)\sigma^d k^{-d/2} n/\log n = \Theta(k^{\beta-d/2-1}/\log n)$ points within $\sigma k^{-1/2}$ of $x$. By Lemma~\ref{lem-Ex}, for all $z \in B_{\sigma k^{-1/2}}(x)$,  $\mathbb{E}F_s(z) > \mathbb{E}F_s(x) - (ek)^{1/2}$. For such a $z$, $\mathbb{P}(F_s(z) > C_{s+1})$ differs from $\mathbb{P}(F_s(z) > C_{s+1})$ by at most a constant factor.  
        Thus, the probability that any given $x$ is chosen is $p_{s}(x) <k^{-\beta + d/2 + 1 + o(1)} < k^{-3/2+o(1)}$ by the definition of $\beta$. There are $s k$ points in $A_0 \cup A_1 \cup \dots \cup A_s$, so the probability that any given $y \in A_0 \cup A_1 \cup \dots \cup A_s$ is in $A_{s+1}$ is at most $sk*k^{-3/2+o(1)} = O(k^{-1/2 + o(1)})$.  Therefore, the probability that $(A_0 \cup A_1 \cup \dots \cup A_s)\cap A_{s+1} = \emptyset$ is at least $1- \frac{1}{k^{1/2 - o(1)}}$. 

		The probability that this holds for all $s < t$ is $(1- \frac{1}{k^{1/2 - o(1)}})^{t} \approx 1 - \frac{t}{k^{1/2 - o(1)}}$. Since $t$ is $\polylog(k)$, this is at least $1 - \frac{1}{k^{1/2 - o(1)}}$. 
	\end{proof}

    Lemma~\ref{lem:independence} implies that, conditioned on the set $A_t$, $F_t(x) = \sum_{z \in A_t} \mathbbm{1}_{(z,x)}$ is a sum of independent indicators (with no dependence on previous time steps). Therefore, $C_t$ can be bounded using standard concentration bounds as follows:

	\begin{lemma} 
		\label{lem:ct-bound}
		At any step $t = O((\log k)^c)$, assuming the conditions of Lemma \ref{lem:main_At-multidim}, with high probability, $C_t \geq \max_x E[x]$
	
	\end{lemma}
\begin{proof}[Proof of Lemma~\ref{lem:ct-bound}] 

	By Lemma~\ref{lem-n-uniform-random}, for any point $x$, there are $\Omega(n \cdot (\sigma k^{-1}\log n)^d)$ points of $V$ in a radius of $\sigma k^{-1}\log n$ of $x$. By the assumption that $n\geq k^{2+d}$, this is $\Omega(k\log n)$.

For any $y \in B_{\sigma k^{-1}}\log n(x)$, Lemma \ref{lem-Ex} implies:

$$E[y] > E[x] - \sqrt{d/e}\log n$$

	Therefore, if $C_t = E[x]$, then there are $\tilde{\Omega}(k)$ points where $E[y]  > C_t -O(\log n)$. Here we will use Lemma~\ref{lem:independence}, which tells us that each $F_t(y)$ is independent conditioned on $A_t$. Hence, Chernoff type bounds apply; if $\mathbb{P}\left(F_t(y) > E[y] + O(\log n)\right) = \Theta(1)$, then with high probability there are $k$ points that exceed $C_t$.

	Using the loose bound given in \cite{volkova1996refinement}, we can bound $\mathbb{P}(F_t(x) > C_t)$ using the CDF of the normal distribution. For any sum of independent indicators $S$ with mean $\mu$ and variance $\sigma$, the CDF can be approximated as follows:
	
	$$\sup_m \left|\mathbb{P}(S \leq m)  - G\left(\frac{m + 1/2 - \mu}{\sigma}\right)\right| \leq \frac{\sigma + 3}{4 \sigma^3}<\frac{1}{\sigma^2}$$
	Where $G(x) = \Phi(x) - \frac{\gamma}{6 \sigma^3}(x^2-1)\frac{e^{-x^2/2}}{\sqrt{2\pi}}$, and $\gamma = \mathbb{E}[(S-\mu)^3]$ is the skewness. This holds for any $\sigma \geq 10$. 

    We can assume that the variance of $y$, $V[y]^2$, exceeds $(\log k)^2$; otherwise, $E[y] = \sum_{y \in A_t} = k(1-o(1/k))$, so we can assume that $C_t = k$. 
	
	Fix $y \in B_{\sigma k^{-1}\log n}(x)$. From the above equation, we find that for any $t>0$:
	
	$$\mathbb{P}\left(F_t(y) > E[y] + tV[y]-1/2\right) > 1-\left[G(t) + \frac{1}{V[y]^2}\right]$$
Substituting the value of $G$: 
	$$\mathbb{P}\left(F_t(y) > E[y] + tV[y]-1/2\right) > 1-\Phi(t) + \frac{\gamma(t^2-1)}{6\sqrt{2\pi} V[y]^3}e^{-t^2/2} - \frac{1}{V[y]^2}$$
	Here, we will make two approximations. First, the exact value of $\gamma$ is $\sum_{z \in A_t}g(y, z)(1-g(y, z))(1-2g(y, z))$. Therefore, $\gamma > -V[y]^2$, so $\gamma(t^2-1) >  -V[y]^2t^2$. Second, we will substitute the lower tail bound for $1-\Phi(t) \geq \frac{1}{\sqrt{2\pi}}\left(t^{-1} - t^{-3}\right)e^{-t^2/2} \geq \frac{1}{t\sqrt{8\pi}}e^{-t^2/2}$ for $t\geq2$.
	
	This leaves us with:
	$$\mathbb{P}\left(F_t(y) > C_t=E[y] + tV[y]-1/2\right) > \frac{1}{t\sqrt{8\pi}}e^{-t^2/2} - \frac{t^2}{6\sqrt{2\pi} V[x]}e^{-t^2/2} - \frac{1}{V[y]^2}$$
	Setting $t = O(1)$, this occurs with constant positive probability. 
\end{proof}
Finally, we can use the above lemma to relate the probability that a point fires at time $t+1$ to its expected value at time $t$. 
	\begin{lemma} 
	\label{lem:Ex-difference}
	Let $y \in I_j$, and $\hat k = |I_j \cap A_t|$.
	If there exists an $x\in I_j$ such that $\mathbb{E}F_t(x) > \mathbb{E}F_t(y) + \sqrt{6\beta(\hat k-\mathbb{E}F_t(y)) \ln k}$, then $\mathbb{P}(F_t(y) > C_t) < \frac{1}{n^3}$
	\end{lemma}
	\begin{proof}
		
		Let $X = \hat k - F_t(y)$. By Lemma~\ref{lem-chernoff}, $\mathbb{P}(X < (1-\epsilon)\mathbb{E}X) \leq \exp\left(-\epsilon^2 \mathbb{E}X/2\right)$.  

		Thus, setting $\epsilon = \frac{C - F_t(y)}{\mathbb{E}X}$, we have 
		$$\mathbb{P}(F_t(y) > C) = \mathbb{P}(X < \mathbb{E}X - (C-\mathbb{E}F_t(y)) \leq  \mathbb{P}(X < \mathbb{E}X (1-\epsilon))\leq \exp\left(-\frac{\epsilon^2 \mathbb{E}X}{2}\right)$$
By Lemma~\ref{lem:ct-bound}, $C_t \geq \mathbb{E}F_t(x)$ for all $x$. Hence, by the assumption, $C-\mathbb{E}F_t(y) \geq \mathbb{E}F_t(x) - \mathbb{E}F_t(y) >  \sqrt{6\beta\mathbb{E}X \ln k}$. Substituting this value for $\epsilon \mathbb{E}X$,

       $$\mathbb{P}(F_t(y) > C) \leq \exp\left(-\frac{6\beta\ln k}{2}\right) = k^{-3\beta} = n^{-3}$$

	\end{proof}

	Now, we are ready to prove Lemma~\ref{lem:main_At-multidim}. This lemma will show that the radius of each ball shrinks at each step; that is $A_t$ is contained within a union of balls of radius $r_t$, where $r_t$ is a decreasing function of $t$. The main idea of the proof is to show that, regardless of the actual positions of points in $A_t \cap I$, vertices toward the center of $I$ have a small advantage over vertices toward the edge. Thus, either (1) the position of points in $A_t \cap I$ is particularly unbalanced, and $A_{t+1}$ shifts toward one side, or (2), the radius of $I$ shrinks in all directions. 
	
\begin{proof}[Proof of Lemma~\ref{lem:main_At-multidim}]
\begin{figure}[h]
    \centering
	\subfloat[\centering Case 1]{{\includegraphics[width=0.46\textwidth]{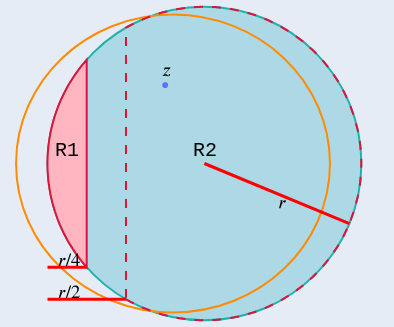} }}
	\qquad
	\subfloat[\centering Case 2]{{\includegraphics[width=0.46\textwidth]{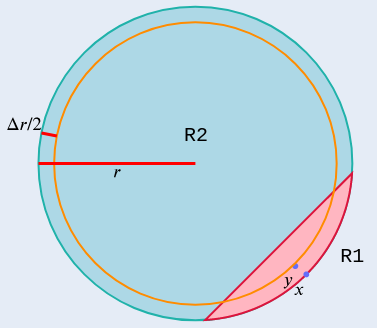} }}
	\caption{The division of the ball $I$ into two subregions. In case 1, there exists a division of $I$ into two sub regions $R_1$ and $R_2$ such that $R_2 \cap A_t < \hat k/(\ln k)^\alpha$. We bound the gradient of $E[z]$ for all $z$ in the region enclosed by the dotted line. In case 2, no such division exists. We prove that for $y$ between the outer and inner circles, $\mathbb{P}(y < 1/n^3)$. In both cases, we prove that $A_{t+1}$ falls in the orange circle with high probability.}
	\label{fig:interval_division}
\end{figure}

Fix one ball $I=B_r(p)$. Let $r$ be the radius of $I$ and $p$ be its center. 

To assist with the proof, we will define the following values. Let $\hat k = |A_t \cap I|$; we can assume that $\hat k > k^{3/4 - o(1)}$, since an interval with asymptotically fewer points will be eliminated at the next step. Define $dist(I,z) = \min_{y \in I} \lVert y-z \rVert$. Finally, for any set $S \subset [0,1]^d$, let $F_t(z; S) = \sum_{y \in A_t \cap S} \mathbbm{1}_{(y,z)}$.

We will prove that with high probability, $\{x \in [0,1]^d : F_t(x; I) \geq C_t\}$ can be covered by $I' = B_{r'}(p')$, where $r' = (1-1/(\log k)^c) r$ and $\max_{z \in I'} dist(I, z) < 5(r(I) - r(I'))$. Call this statement (*).

If statement (*) holds for each $I = I_i$, the lemma is proven. This holds by the separation assumption; if $dist(I, x) < \sigma\sqrt{\ln n}$, then for all $y \in A_t \setminus I$, $g(x, y) < 1/n^{2(1-o(1))}$.  Therefore, $F_t(x; I) = F_t(x) - o(1)$ with high probability.

To prove this statement, we will consider two cases. In case 1, we suppose that the distribution of $A_t \cap I$ is imbalanced. In particular, there exists a half space dividing $I$ into two spherical caps, with heights $r/4$ and $7r/4$, such that the larger segment contains only $\hat k/(\ln k)^\alpha$ points of $A_{t+1}$ (for an $\alpha \geq 1$). See Figure \ref{fig:interval_division}. We will show that given this imbalance, statement (*) holds. 

In case 2, no such division exists. We will show that for any point $z$ near the boundary of $I$ is disadvantaged compared to a point near the center. Thus, $I' = B_{r'}(p)$ for an $r' = (1-1/(\log k)^c) r$. See Figure \ref{fig:interval_division}.

For both cases, the argument will use a bound on the gradient of $E[z]$. With this, we will construct a point $w$ such that  $E[w] - E[z]$ is large, and use Lemma \ref{lem:Ex-difference} to argue that $\mathbb{P}(z \in A_{t+1}) < 1/n^3$.

Consider two cases:

\noindent
\textbf{Case 1}: there exists a half space dividing $I$ into two spherical caps, $R_1$ and $R_2$, with heights $r/4$ and $7r/4$, such that $|R_2 \cap A_t| \leq \hat k/(\ln k)^{\alpha}$, where $\alpha = 1+\max(2r^2/(\sigma^2\ln \ln k), 1)$. In this case, the ball is ``imbalanced" in the sense that one portion of of the ball contains the vast majority of the points. 

Without loss of generality, let $p = [r,0,0, \dots, 0]$, $R_1 = \{y \in I : y_1 \leq r/4\}$, and $R_2 = \{y \in I : y_1 > r/4\}$ (as illustrated in Figure \ref{fig:interval_division}). Let $z = [z_1, z_2, \dots, z_n]$ where $z_1 \geq 3r/8$ and $dist(I, z) = O(\min\{r, \sigma \hat{k}^{-1/5}\})$. 

Then we can bound the derivative with respect to the first coordinate:

\begin{align}
\frac{\partial}{\partial z_1}\mathbb{E}F_t(z; I)& = \sum_{y \in A_t\cap I} \frac{\partial}{\partial z_1} g(y,z)=\sum_{y \in A_t\cap I} -\frac{z_1 - y_1}{\sigma^2}g(y,z)\\
&= \sum_{y \in R_1 \cap A_t} -\frac{z_1 - y_1}{\sigma^2}g(y,z) + \sum_{y \in R_2 \cap A_t} -\frac{z_1 - y_1}{\sigma^2}g(y,z)\label{eq-derivative}
\end{align}
The partial derivative $ \frac{\partial}{\partial z_1}g(y,z)$ is minimized at $z_1 - y_1 = \sigma$ and maximized at $z_1 - y_1 = -\sigma$. The lower bound on the derivative depends on $r$ as follows:
\begin{itemize} 
\item If $2r \geq \sigma$:

$\min_{y \in R_2}  \frac{z_1 - y_1}{\sigma^2}g(y,z) > -\frac{1}{\sigma}e^{-1/2}$, and

$\min_{y \in R_1} \frac{z_1 - y_1}{\sigma^2}g(y,z)>\frac{r}{8\sigma^2}\exp\left(\frac{-(2r + \min\{r, \sigma \hat k^{-1/5}\})^2}{2\sigma^2}\right) = \frac{r}{8\sigma^2}\exp\left(\frac{-2r^2 - o(\sigma^2)}{\sigma^2}\right)$

\item If $2r< \sigma$, 

$\min_{y \in R_2}  \frac{z_1 - y_1}{\sigma^2}g(y,z) > -\frac{2r}{\sigma^2}\exp\left(-2r^2/\sigma^2\right)$, and

$\min_{y \in R_1} \frac{z_1 - y_1}{\sigma^2}g(y,z)>\frac{r}{8\sigma^2}\exp\left(\frac{-(2r + O(\min\{r, \sigma \hat k^{-1/5}\})^2}{2\sigma^2}\right) = \frac{r}{8\sigma^2}\exp\left(\frac{-O(1)r^2}{2\sigma^2}\right)$
\end{itemize}
Returning to Equation \ref{eq-derivative}, 
\begin{equation}
\frac{\partial}{\partial z_1}\mathbb{E}F_t(z; I)\leq -|A_t \cap R_1|\min_{y \in R_1}  \frac{z_1 - y_1}{\sigma^2}g(y,z)  + |A_t \cap R_2|\max_{y \in R_2}  \frac{y_1-z_1}{\sigma^2}g(y,z) 
\end{equation}
By assumption, $|A_t\cap R_2|\leq \hat k/(\ln k)^{\alpha}$. Replacing this:
\begin{itemize} 
\item If $2r \geq \sigma$:
$$\frac{\partial}{\partial z_1}\mathbb{E}F_t(z; I)\leq -\hat k(1-o(1))\frac{r}{8\sigma^2}\exp\left(\frac{-2r^2}{\sigma^2}-o(1)\right)+ \frac{\hat k}{(\ln k)^{\alpha}}\frac{1}{\sigma}e^{-1/2}$$ 
By the definition of $\alpha$, $(\ln k)^{-\alpha} = e^{-\alpha \ln \ln k}\leq \frac{1}{\ln k}e^{-2r^2/\sigma^2}$. Hence,
$$\frac{\partial}{\partial z_1}E[z]\leq -\frac{\hat k}{\sigma} \left[(1-o(1))\frac{r}{8\sigma}\exp\left(\frac{-2r^2}{\sigma^2}\right) -\frac{1}{\ln k}\exp\left(\frac{-2r^2}{\sigma^2}\right)e^{-1/2}\right]$$

$$\frac{\partial}{\partial z_1}\mathbb{E}F_t(z; I)\leq -\frac{\hat kr}{8\sigma^2}\exp\left(\frac{-2r^2}{\sigma^2}\right) (1-o(1))$$

\item If $2r< \sigma$, 

$$\frac{\partial}{\partial z_1}\mathbb{E}F_t(z; I)\leq - \frac{r}{8\sigma^2}e^{-O(1)r^2/\sigma^2} \hat k + \frac{2r}{\sigma^2}\frac{\hat k}{\ln k} = -\Theta(1)\frac{\hat kr}{\sigma^2}$$
\end{itemize}

Let $z' = [z_1', z_2', \dots, z_n']$ where $z_1' \geq r/2$ and $dist(I, z') < \min\{r/8, \sigma \hat k^{-1/5}\}$. Consider the point $w = z' - [\min\{r/8, \sigma \hat k^{-1/5}\}, 0, 0, \dots, 0]$. By definition, the derivative bounds above apply for all points on the line between $w$ and $z'$. This gives us a lower bound on $E[w] - E[z']$. While $w\notin V$ almost surely, by Lemma \ref{lem-n-uniform-random} there exists a point $w' \in V$ within a radius of $O((\log n/n)^{1/d})$ of $w$. Applying \ref{lem-Ex}, $E[w'] - E[z'] > E[w] - E[z'] - o(1)$. Then, we will apply Lemma \ref{lem:Ex-difference} to show that $\mathbb{P}(z' \in A_{t+1}) < 1/n^3$.

The condition of Lemma~\ref{lem:Ex-difference} holds if 
\begin{equation}
\label{eq:cond-ex-difference}
E[w] - E[z'] \geq \sqrt{6\beta(\hat{k} - E[z])\ln k}
\end{equation}

\begin{itemize}

\item If $2r \geq \sigma$:
$$E[w] \geq E[z'] +\sigma \hat k^{-1/5}\cdot \frac{\hat kr}{8\sigma^2}\exp\left(\frac{-2r^2}{\sigma^2}\right) (1-o(1))$$
Since $r = O(\sigma \sqrt{\ln \ln k})$, $\exp\left(2r^2/\sigma^2\right) = \tilde{O}(1)$. Thus, $E[w] - E[z]= \tilde{\Omega}(\hat k^{4/5})$. Clearly this exceeds $\sqrt{6\beta\hat k\ln k}$, so by Lemma \ref{lem:Ex-difference}, $\mathbb{P}(z' \in A_{t+1}) < 1/n^3$.

\item If $20 \sigma \hat k^{-1/5} < 2r < \sigma$:

For the same reasons as above, we can obtain a similar bound:

$$E[w] \geq E[z'] + \sigma \hat k^{-1/5}\cdot \hat k\frac{r}{6\sigma^2}e^{-2r^2/\sigma^2}(1-o(1)) = E[z] + \hat k^{4/5 - o(1)} \frac{r}{\sigma}$$
Since $r= \Omega(\sigma \hat k^{-1/5})$, this exceeds $\sqrt{6\beta\hat k\ln k}$, so by Lemma \ref{lem:Ex-difference}, $\mathbb{P}(z' \in A_{t+1}) < 1/n^3$.

\item If $2r \leq 20 \sigma \hat k^{-1/5}$:
$$E[w] \geq E[z'] + \frac{r}{8}\cdot \hat k\frac{r}{8\sigma^2}e^{-2r^2/\sigma^2}(1-o(1)) = E[z] + \Theta(1)\hat k\frac{r^2}{\sigma^2}$$

In this case, we can bound $\hat k-E[z]$; 
$$\hat k - E[z] \leq \hat k(1 - e^{-2r^2/\sigma^2})\leq \hat k \frac{2r^2}{\sigma^2}$$

Therefore, the condition can be bounded: $\sqrt{6\beta(\hat k - E[z])\ln k}\leq \frac{r}{\sigma} \sqrt{12\beta \hat k\ln k}$. 

There exists a constant $C$ such that for $10\sigma \hat k^{-1/5} > r > C\sigma \sqrt{\ln k/\hat k}$, 
$E[w] - E[z] = \Theta(1)\hat k \frac{r^2}{\sigma^2} > \frac{r}{\sigma} \sqrt{12\beta\hat k\ln k}$. 
By Lemma \ref{lem:Ex-difference}, $\mathbb{P}(z' \in A_{t+1}) < 1/n^3$.
\end{itemize}

Finally, we will argue that for any $z$ with $dist(I, z) > r/20$,  $\mathbb{P}(z \in A_{t+1}) < 1/n^3$. Let $u$ be the unit vector parallel to $z-p$:
\begin{align*}
\nabla_u E[z]&=\sum_{y \in A_t} \nabla_u g(y,z) =\sum_{y \in A_t} \left(u \cdot \frac{y-z}{\sigma^2}\right)g(y,z)\\
&\geq \frac{dist(I, z)}{\sigma^2}\sum_{y \in A_t}g(y,z) \\
&= \frac{dist(I, z)}{\sigma^2}E[z]\\
\end{align*}
Let $w$ be a point along the line $z-p$, with $dist(I, w) = dist(I,z)/2$. Again, while $w\notin V$ almost surely, by Lemma \ref{lem-n-uniform-random} there exists a point $w' \in V$ within a radius of $O((\log n/n)^{1/d})$ of $w$. Applying \ref{lem-Ex}, $E[w'] - E[z'] > E[w] - E[z'] - o(1)$. 
Dividing this again into two cases:
\begin{itemize}
\item If $r\geq 2\sigma$:

There exists a point $y$ in $I$ with $E[y] =\Omega( \hat k / (\ln \ln k)^{d/2})$.  This is due to the pigeonhole principle; the volume of $I$ is $\Theta(r^d) = O((\ln \ln k)^{d/2}/\hat k)$. Therefore, there exists a smaller ball of radius $\sigma$ in $I$ with $\hat k/(\ln \ln k)^{d/2}$ points. For $y$ in this smaller ball, $E[y] = \Omega( \hat k / (\ln \ln k)^{d/2})$.

If $E[z] = \tilde{\Omega}(\hat k)$, then $E[w] - E[z] > dist(I,z)^2/(2\sigma)^2 E[z] = \tilde{\Omega}(\hat k)$, and by Lemma \ref{lem:Ex-difference}, $\mathbb{P}(z \in A_{t+1}) < 1/n^3$. Otherwise, $E[y] - E[z] = \tilde{\Omega}(\hat k)$, and again by Lemma \ref{lem:Ex-difference}, $\mathbb{P}(z \in A_{t+1}) < 1/n^3$.

\item If $r<2 \sigma$:

There exists a point $y$ in $R_1$ with $E[y] \geq e^{-r^2/8\sigma^2}\hat k \geq \hat k(1-r^2/8\sigma^2)$. 

If $E[z] = \hat k(1-\gamma r^2/\sigma^2)$,  the bound for Lemma \ref{lem:Ex-difference} is: $$\sqrt{6\beta(\hat k -E[z])\ln k} = \frac{r}{\sigma} \sqrt{6\beta \gamma \ln k} $$

For any $r = \Omega(\sigma\sqrt{\ln k/\hat k})$, $E[w] - E[z] \geq \Theta(r^2/\sigma^2\hat k) = \Omega(\ln k)$. There exists a constant $C$ such that for any $r>C\sigma\sqrt{\ln k/\hat k}$, this exceeds the bound of Lemma \ref{lem:Ex-difference}, and $\mathbb{P}(z \in A_{t+1}) < 1/n^3$.
\end{itemize}

In conclusion, we have determined that the set of points $z\in I$ such that $\mathbb{P}(z \in A_{t+1})>1/n^3$ are contained within a region $R = \{z : z_1 \leq r/2, dist(I, z) < r/20\}$. The radius of $R$ is $r/20$ plus the width of $\{z\in I : z_1 \leq r/2\}$. Using a geometric argument, this set has width $\sqrt{R^2 - (R/2)^2} = R\sqrt{3}/2$. This region can be enclosed by a ball $I'$ defined as follows (illustrated as an orange circle in Figure \ref{fig:interval_division}):

Let $I' =B_{19r/20}(p')$ for $p' = [3r/4, 0, 0, \dots, 0]$ and $r(I') = \frac{19}{20}r(I)$. It is simple to check that $B_{19r/20}(p')$ contains $R$; 
$$\max_{z \in R} \lVert p-z \rVert = r\sqrt{(\sqrt{3}/2 + 1/20)^2 + 1/4^2} < 19r/20$$

Additionally, $\max_{z \in B_{19r/20}(p')} dist(z, I) < \max_{z \in B_{r}(p')} dist(z, I) < r/4$. Therefore, $d(I, I') < r/4 < 5(r(I) - r(I'))$.

\noindent
\textbf{Case 2}: In this section, we assume that no such imbalanced partition of $I$ exists. For all $x \in \partial I$, denoting $R_2 = \{z\in I : (x-z) \cdot \frac{x-p}{\lVert x-p \rVert} > r/4\}$, $|R_2 \cap A_t|\geq \hat k/(\ln k)^{\alpha}$, where $\alpha = 1+\max\{2r^2/(\sigma^2\ln \ln k), 1\}$. We will show that for any $z$ within $\Delta r$ of the boundary of $I$, $\mathbb{P}(z \in A_{t+1} < 1/n^3)$.

Fix $x$, and assume without loss of generality that $p = [r,0,0,\dots, 0]$ and $x = [0,0,\dots, 0]$

Let $z = [z_1,0,0,\dots,0]$ where $-\sigma\log n < z_1 < \Delta r = O(r/\ln k)$. Then, $\max_{y \in R_1} \frac{z_1 - y_1}{\sigma^2}g(y,z) = \frac{z_1}{\sigma^2}\exp\left(\frac{-z_1^2}{2\sigma^2}\right)< \frac{\Delta r}{\sigma^2}\exp\left(\frac{-(\Delta r)^2}{2\sigma^2}\right)$ (Note that by construction $\Delta r<\sigma$). Also, separately taking the minima of $y_1 - z_1$ and $g(y,z)$, $\min_{y \in R_2} \frac{y_1 - z_1}{\sigma^2}g(y,z) > \frac{r/4-z_1}{\sigma^2}\exp\left(\frac{-2r^2}{\sigma^2}\right)$. Returning to Equation \ref{eq-derivative}:

\begin{equation}
\frac{\partial}{\partial z_1}E[z]\geq -|A_t \cap R_1| \frac{\Delta r}{\sigma^2}\exp\left(\frac{-(\Delta r)^2}{2\sigma^2}\right) + |A_t \cap R_2|\frac{r/4-z_1}{\sigma^2}\exp\left(\frac{-2r^2}{\sigma^2}\right)
\end{equation}
By assumption, $|A_t \cap R_2| \geq \hat k/(\ln k)^{\alpha}$. Replacing this:
$$\frac{\partial}{\partial z_1}E[z]\geq -\hat k\left[1-\frac{1}{(\ln k)^{\alpha}}\right] \frac{\Delta r}{\sigma^2}\exp\left(\frac{-(\Delta r)^2}{2\sigma^2}\right) + \frac{\hat k}{(\ln k)^{\alpha}} \frac{r/4-\Delta r}{\sigma^2}\exp\left(\frac{-2r^2}{\sigma^2}\right)$$

Define $\Delta r = r/(\ln k)^{2\alpha}$. Again, note that $(\ln k)^{-\alpha} = e^{-\alpha \ln \ln k}\leq \frac{1}{\ln k}e^{-2r^2/\sigma^2}$. So, $e^{-2r^2/\sigma^2} \geq (\ln k)^{1-\alpha}$
$$\frac{\partial}{\partial z_1}E[z]\geq -\hat k\left[1-\frac{1}{(\ln k)^{\alpha}}\right] \frac{r}{\sigma^2 (\ln k)^{2\alpha}}\exp\left(\frac{-(\Delta r)^2}{2\sigma^2}\right) + \frac{\hat k}{(\ln k)^{2\alpha-1}} \frac{r}{5\sigma^2}$$
$$\frac{\partial}{\partial z_1}E[z]\geq \frac{\hat kr}{\sigma^2 (\ln k)^{2\alpha}}\left[1-o(1) + \ln k\right]$$

Suppose $z = [z_1, 0, 0, \dots, 0]$ where $z_1 < \Delta r/2$. Let $w = z_1 + [\Delta r/2, 0, 0, \dots, 0]$. Using the lower bound on the derivative, 

$$E[w] \geq E[z] + \frac{\Delta r}{2}\cdot\frac{\hat kr}{\sigma^2 (\ln k)^{2\alpha}}\left[1-o(1) + \ln k\right] \geq E[z] + \frac{1}{2}\frac{\hat kr^2}{\sigma^2}\ln k$$

For $r= \Omega(\sigma \hat k^{-1/4})$, this exceeds $\sqrt{6\beta\hat k\ln k}$, so by Lemma \ref{lem:Ex-difference}, $\mathbb{P}(z \in A_{t+1}) < 1/n^3$.

For $r = o(\sigma \hat k^{-1/4})$, we can bound $\hat k-E[z]$; 
$$\hat k - E[z] \leq \hat k(1 - e^{-2r^2/\sigma^2})\leq \hat k \frac{2r^2}{\sigma^2}$$

Therefore, the condition can be bounded: $\sqrt{6\beta(\hat k - E[z])\ln k}\leq \frac{r}{\sigma} \sqrt{12\beta \hat k\ln k}$. 

Then, for $\sigma \hat k^{-1/4} \ln k > r > \sigma \sqrt{\ln k/\hat k}$, 
$E[w] - E[z] = \frac{1}{2}\frac{\hat kr^2}{\sigma^2}\ln k > \frac{r}{\sigma} \sqrt{12\beta \hat k\ln k}$. 
By Lemma \ref{lem:Ex-difference}, $\mathbb{P}(z' \in A_{t+1}) < 1/n^3$.

In summary, there are two cases: in \textbf{case 1}, there exists a partition of $I$ such that the vast majority of $A_t$ is located in $R_1$. In this case, we have shown that $\{z : \mathbb{P}(z \in A_{t+1}) > 1/n^3\} \subset \{z : z_1 \leq r/2\}$. A symmetric argument showed that for any $z$ such that the distance from $z$ to $I$ is at most $r/20$, $\mathbb{P}(z \in A_{t+1}) < 1/n^3$ Therefore, $A_{t+1}$ is contained within a ball $I'= B_{19r/20}(p')$, where $d(I, I')< r/4$

In \textbf{case 2}, for all $x \in \partial I$, $|A_t \cap R_2|$ is sufficiently large. In this case, We have shown that for all $z = x + \lambda \frac{p-x}{\lVert p-x \rVert}$ where $\lambda < r/\polylog(k)$, $\mathbb{P}(z \in A_{t+1}) < 1/n^3$. This applies for all $x\in \delta I$. Therefore, $I' \subset B_{r(1 - 1/(\ln k)^c)}(p)$.

By the assumption that each ball is sufficiently separated, we can conclude that, with high probability, $A_{t+1} \subset I_1' \cup I_2' \cup \dots \cup I_i'$, where the radius of $I_j'$ is smaller than the radius of $I_j$ by at least a factor of $1/\polylog(k)$.

\end{proof}

\noindent Now, we are ready to prove Theorem~\ref{thm-main-at-multidim}.
\begin{proof}[Proof of Theorem~\ref{thm-main-at-multidim}] 
    By Theorem \ref{thm:A1}, the conditions of Lemma~\ref{lem:main_At-multidim} hold at step 1. Additionally, by the condition that each ball does not shift by more than $5(r(I) - r(I'))$ at each step, the separation condition holds inductively for any $t = \polylog(k)$. The maximum distance moved by a single ball by time $t$ is $5 (r(I^{(0)}_j) - r(I^{(t)}_j)) = O(\sigma \sqrt{\ln \ln k})$, which maintains the separation of $2(1-o(1))\sigma \sqrt{\ln n})$. Thus, we can apply the Lemma inductively. 
    
	By Lemma~\ref{lem:main_At-multidim}, the radius of $I$ is reduced by a factor of $1-\frac{1}{\polylog(k)}$ in a single step; thus, to reach $O(\sigma \hat k^{-1/2}\sqrt{\ln k})$, the number of steps required is $\polylog(k)$. 
 
	This shows that in $\polylog(k)$ time, the radius of each sufficiently separated ball will be reduced to at most $\hat k^{-1/2}\sqrt{\ln k}$. Recall that there are $k^{1/4+o(1)}$ separated balls; a similar method will allow us to eliminate balls that are $\sigma \hat k^{-1/2}\sqrt{\ln k}$ in size.

	If the number of balls is greater than 1, $|I_1\cap A_t|$ can fall anywhere in the range $M \pm \sqrt{M}$ with constant probability, where $M = \mathbb{E}|I_1\cap A_t|$. By the pigeonhole principle, at least one ball receives $k^{3/4 - o(1)}$ points. Since the size of the balls are at most $\sigma \hat k^{-1/2}\sqrt{\ln k} < \sigma \hat k^{-3/8}\sqrt{\ln k}$, $C_t\geq \left(k^{3/4 - o(1)}\right)$. Therefore, for each $j$, if $I_j$ is not eliminated, there exists an $x \in I_j$ such that $\mathbb{E}F_t(x) \geq k^{3/4 - o(1)}$. Consider two alternative scenarios, which can each occur with constant probability.  
	$$
	\begin{cases}
	(1)  & \max_{x \in I_1}\mathbb{E}F_t(x) = X - \Theta(\sqrt{X})\\
	(2)  & \max_{x \in I_1}\mathbb{E}F_t(x) = X + \Theta(\sqrt{X})
	\end{cases}
	$$
	Let $y = \argmax_{y \in I_2}\mathbb{E}F_t(y)$. So, it is clear that in either scenario (1) or (2), the inputs to $x$ and $y$ differ by the number of points added to $I_1$. 
	$$\left|\max_{x \in I_1}\mathbb{E}F_t(x) -\max_{y \in I_2}\mathbb{E}F_t(y)\right| = k^{3/8 - o(1)}$$
	
	In scenario 2, $x$ receives an extra input of $\Theta\left(\sqrt{X}\right)$. The increased input in this scenario could affect $C_t$; however, either $\mathbb{E}F_t(x)$ becomes closer to $C_t$ by $k^{3/8 - o(1)}$, or $\mathbb{E}F_t(y)$ becomes further from $C_t$ by the same amount.

	By Lemma~\ref{lem:pt-lowerbound}, between the two scenarios, either $p_{t+1}(z)$ increases by a constant factor for all $z \in I_1$, or $p_{t+1}(w)$ decreases by a constant factor for all $w \in I_2$. Again, $I_1$ and $I_2$ either have at least $k^{3/4 - o(1)}$ points, or they are eliminated. So, this implies that  $|\mathbb{E}F_{t+1}(x)-\mathbb{E}F_{t+1}(y)|$ varies by $O(1)k^{3/4 - o(1)}$ between the two scenarios. 
	
	This is a significant variation; as in Lemma~\ref{lem:ct-bound}, for any $x \in [n]$, $C_t \geq \mathbb{E}F_{t+1}(x)$. So, in the case where $\mathbb{E}F_{t+1}(y)< \mathbb{E}F_{t+1}(x)$:
	
	$$\mathbb{E}F_{t+1}(y)< \mathbb{E}F_{t+1}(x) -\Omega\left(k^{3/4 - o(1)}\right)$$
	
	By the Chernoff bound, the probability that $F_{t+1}(y)$ will exceed $C_{t+1}$ is exponentially small. A similar argument applies if $\mathbb{E}F_{t+1}(y) > \mathbb{E}F_{t+1}(x)$. Therefore, since there is a constant probability that the two balls will deviate from each other, either $I_1$ or $I_2$ will be eliminated in a constant number of steps. 
	
	The same argument applies to any pair of balls $(I_i, I_j)$. Therefore, the number of balls reduces by a constant factor within a constant number of steps. This leads to convergence to a single ball within $O(\ln k)$ steps. 

    At this point, $\hat k = k$, so applying Lemma \ref{lem:main_At-multidim} again, we can conclude that $A_t$ converges to a single ball of size $O(\sigma k^{-1/2}\sqrt{\ln k})$ in $O((\log k)^c)$ steps. 
\end{proof}

\noindent Finally, we prove that the set $A_t$, with high probability, remains within a small subset for all $t \geq t^*$.
\begin{proof} [Proof of Theorem \ref{thm-all-t-multidim}]
Let $A\subset V$ with $|A| = k$, and let $I$ be a ball surrounding $k - k^{2/3}$ points of $A$. Assume that $r = r(I) = \sigma k^{-1/3+\epsilon}$ and $I=B_r(p)$. 

We consider 2 cases, identically to the proof of Lemma \ref{lem:main_At-multidim}:

\noindent
\textbf{Case 1:} There exists a half space dividing $I$ into two spherical caps, $R_1$ and $R_2$, with heights $r/4$ and $7r/4$, such that $|R_2 \cap A_t| \leq k/(\ln k)^{2}$. In this case, the ball is ``imbalanced" in the sense that one portion of of the ball contains the vast majority of the points. 

Without loss of generality, let $p = [r,0,0, \dots, 0]$, $R_1 = \{y \in I : y_1 \leq r/4\}$, and $R_2 = \{y \in I : y_1 > r/4\}$ (as illustrated in Figure \ref{fig:interval_division}). We define $E[z] = F(z; A) = \sum_{y \in A} g(y,z)$.

Then we can bound the derivative with respect to the first coordinate:
\begin{align}
\frac{\partial}{\partial z_1}E[z]& = \sum_{y \in A\cap I} \frac{\partial}{\partial z_1} g(y,z)=\sum_{y \in A\cap I} -\frac{z_1 - y_1}{\sigma^2}g(y,z) + \sum_{y \in A\setminus I} -\frac{z_1 - y_1}{\sigma^2}g(y,z) \\
&= \sum_{y \in R_1 \cap A} -\frac{z_1 - y_1}{\sigma^2}g(y,z) + \sum_{y \in R_2 \cap A} -\frac{z_1 - y_1}{\sigma^2}g(y,z) + \sum_{y \in A \setminus I} -\frac{z_1 - y_1}{\sigma^2}g(y,z)\label{eq-derivative}
\end{align}
Let $z = [z_1, z_2, \dots, z_n]$ where $z_1 \geq 3r/8$ and $dist(I, z) = O(r)$. This implies: 

$\min_{y \in R_2}  \frac{z_1 - y_1}{\sigma^2}g(y,z) > -\frac{2r}{\sigma^2}\exp\left(-2r^2/\sigma^2\right)$

$\min_{y \in R_1} \frac{z_1 - y_1}{\sigma^2}g(y,z)> \frac{r}{8\sigma^2}\exp\left(\frac{-O(1)r^2}{2\sigma^2}\right)$, and

$\min_{y \in [0,1]^d\setminus I} \frac{z_1 - y_1}{\sigma^2}g(y,z)> -\frac{1}{\sigma}e^{-1/2}$
\\Returning to Equation \ref{eq-derivative}:
\begin{equation}
\frac{\partial}{\partial z_1}E[z]\leq -|A_t \cap R_1|\min_{y \in R_1}  \frac{z_1 - y_1}{\sigma^2}g(y,z)  - |A_t \cap R_2|\min_{y \in R_2}  \frac{z_1-y_1}{\sigma^2}g(y,z) - |A_t \setminus I| \min_{y \in [0,1]^d}  \frac{z_1-y_1}{\sigma^2}g(y,z) 
\end{equation}
By assumption, $|A_t\cap R_2|\leq k/(\ln k)^{2}$ and $|A_t \setminus I|=k^{2/3}$. Replacing this:

$$\frac{\partial}{\partial z_1}E[z]\leq - \frac{r}{8\sigma^2}e^{-O(1)r^2/\sigma^2} k+ \frac{2r}{\sigma^2}\frac{k}{\ln k^2} + \frac{\Theta(1) k^{2/3}}{\sigma} = -\Theta(1)\frac{kr}{\sigma^2}$$

Let $z' = [z_1', z_2', \dots, z_n']$ where $z_1' \geq r/2$ and $dist(I, z') < r/8$. Consider the point $w = z' - r/8$. By definition, the derivative bounds above apply for all points on the line between $w$ and $z'$. This gives us a lower bound on $E[w] - E[z']$. While $w\notin V$ almost surely, by Lemma \ref{lem-n-uniform-random} there exists a point $w' \in V$ within a radius of $O((\log n/n)^{1/d})$ of $w$. Applying \ref{lem-Ex}, $E[w'] - E[z'] > E[w] - E[z'] - o(1)$. 

We will prove an analogous result to Lemma \ref{lem:ct-bound} for sets $A$ contained mostly within a ball of radius $r$.

\begin{lemma} 
		\label{lem:ct-bound-small}
		Let $I$ be a ball of radius $r=\sigma k^{-1/3 + \epsilon}$, for some constant $\epsilon >0$, surrounding $k-k^{2/3}$ points of $A$. Let $C$ be the threshold when the $k$-cap function is applied to $A$. With high probability, for all such sets $A \subset V$ with $|A| = k$, $C \geq \max_x E[x]$.

	\end{lemma}
\begin{proof}[Proof of Lemma~\ref{lem:ct-bound}] 

 The derivative of $E[x]$ is, for any dimension $i$:
 $$\frac{\partial}{\partial x_i} E[x]= \sum_{z \in A_t}\frac{\partial}{\partial x_i}g(x, z) = \sum_{z \in A_t} -\frac{x_i-z_i}{\sigma^2} \exp\left(\frac{-\lVert  x-z \rVert^2}{2\sigma^2}\right)$$
	Let $\hat{x} = \frac{x-z}{\sigma}$. The maximum of $\left|\hat{x} \exp\left(\frac{-\hat{x}^2}{2}\right)\right|$ occurs at $e^{-1/2}$. For $y \in I$, this is maximized at $\hat x = \frac{2r}{\sigma}$. Thus, we have 
$$ \left|\frac{\partial}{\partial x_i}  E[x]\right| \leq \sum_{z \in A_t}\frac{|x_i-z_i|}{\sigma^2} \exp\left(\frac{-\lVert x-z\rVert^2}{2\sigma^2}\right)<\frac{e^{-1/2}k^{2/3}}{\sigma}  + k \frac{2r}{\sigma^2} = \frac{2k^{2/3+\epsilon}}{\sigma}(1+o(1))$$
Therefore, the directional derivative, as in Lemma \ref{lem-Ex}, is at most this value, times a factor of $\sqrt{d}$. 
For any $y \in B_{\sigma k^{-1+\epsilon}}(x)$, the difference between $E[y]$ and $E[x]$ can be bounded:

$$E[y] = E[x] - o(1)$$

By Lemma~\ref{lem-n-uniform-random}, for any point $x$, there are $\Omega((n/\log n) \cdot \sigma^d k^{-d+d\epsilon})$ points in a radius of $\sigma k^{-1 + \epsilon}$ of $x$. By the assumption that $n\geq k^{2+d}$, this is $\tilde{\Omega}(k^{1+d\epsilon})$.
 
	Therefore, if $C = E[x]$, then there are $\tilde{\Omega}(k^{1+d\epsilon})$ points where $E[y]  > C_t -O(1)$. Since each edge is chosen independently, Chernoff type bounds apply; if $\mathbb{P}\left(F_t(y) > E[y] + O(1)\right) = \Theta(1)$, then with high probability there are $k$ points that exceed $C_t$.

	Using the loose bound given in \cite{volkova1996refinement}, we can bound $\mathbb{P}(F_t(x) > C_t)$ using the CDF of the normal distribution. For any sum of independent indicators $S$ with mean $\mu$ and variance $\sigma$, the CDF can be approximated as follows:
	
	$$\sup_m \left|\mathbb{P}(S \leq m)  - G\left(\frac{m + 1/2 - \mu}{\sigma}\right)\right| \leq \frac{\sigma + 3}{4 \sigma^3}<\frac{1}{\sigma^2}$$
	Where $G(x) = \Phi(x) - \frac{\gamma}{6 \sigma^3}(x^2-1)\frac{e^{-x^2/2}}{\sqrt{2\pi}}$, and $\gamma = \mathbb{E}[(S-\mu)^3]$ is the skewness. This holds for any $\sigma \geq 10$. 

    We can assume that $V[y] > 10$; otherwise, $E[y] = \sum_{y \in A_t} = k(1-o(1/k))$, so we can assume that $C = k$. 
	
	Fix $x \in [n]$. From the above equation, we find that for any $t>0$:
	
	$$\mathbb{P}\left(F_t(x) > E[x] + tV[x]-1/2\right) > 1-\left[G(t) + \frac{1}{V[x]^2}\right]$$
	
	Substituting the value of $G$: 
	$$\mathbb{P}\left(F_t(x) > E[x] + tV[x]-1/2\right) > 1-\Phi(t) + \frac{\gamma(t^2-1)}{6\sqrt{2\pi} V[x]^3}e^{-t^2/2} - \frac{1}{V[x]^2}$$
	
	Here, we will make two approximations. First, the exact value of $\gamma$ is $\sum_{z \in A_t}g(x, z)(1-g(x, z))(1-2g(x, z))$. Therefore, $\gamma > -V[x]^2$, so $\gamma(t^2-1) >  -V[x]^2t^2$
	
	Second, we will substitute the lower tail bound for $1-\Phi(t) \geq \frac{1}{\sqrt{2\pi}}\left(t^{-1} - t^{-3}\right)e^{-t^2/2} \geq \frac{1}{t\sqrt{8\pi}}e^{-t^2/2}$ for $t\geq2$.
	
	This leaves us with:
	$$\mathbb{P}\left(F_t(x) > E[x] + tV[x]-1/2\right) > \frac{1}{t\sqrt{8\pi}}e^{-t^2/2} - \frac{t^2}{6\sqrt{2\pi} V[x]}e^{-t^2/2} - \frac{1}{V[x]^2}$$
	Setting $t = \Theta(1)/V[x]$, this occurs with constant positive probability $p$.

    The probability that there are not $k$ points which exceed $C=E[x]$ is at least ${k^{1+d\epsilon}\choose k}(1-p)^{k^{1+d/2}-k} = (1-p)^{k^{1+d\epsilon}(1-o(1))}$. 

    The number of possible subsets $A$ is at most ${n \choose k} < n^k = e^{k \log n}$. 

    By the union bound, this holds for all subsets $A$ with high probability. 
\end{proof}

Returning to the proof of the original theorem, we recall that there exists a point $w$ such that:
$$E[w] \geq E[z'] + \frac{r}{8}\cdot k\frac{r}{8\sigma^2}e^{-2r^2/\sigma^2}(1-o(1)) = E[z] + \Theta(1)k\frac{r^2}{\sigma^2} = E[z'] + \Theta(k^{1/3 + 2\epsilon})$$

Comparing this to $k-E[z']$
$$k - E[z'] \leq  k(1 - e^{-2r^2/\sigma^2})\leq k \frac{2r^2}{\sigma^2} = 2k^{1/3+2\epsilon}$$

We can apply Lemma \ref{lem-chernoff} to $k-E[z']$; let $Z = k-F(z'; A)$. Then, $\mathbb{P}(Z < (1-\delta)\mathbb{E}Z])\leq e^{-\delta^2 \mathbb{E}Z/2}$. So, 
$$\mathbb{P}(Z <\mathbb{E}Z - \Theta(1)\mathbb{E}Z) < e^{-\Theta(1) k^{1/3+2\epsilon}} $$

The probability that there exist $k^{2/3}$ points which violate the condition is at most:

$$ {n \choose k^{2/3}}(e^{-\Theta(1) k^{1/3+2\epsilon}})^{k^{2/3}} < e^{k^{2/3}\log n}e^{-\Theta(1)k^{1+2\epsilon} \log n}$$

Since there are at most ${n \choose k} = O(e^{k \log n})$ possible $k$-subsets of $V$, this is true by the union bound for all subsets $A$ with high probability.

\noindent \textbf{Case 2}: In this section, we assume that no such imbalanced partition of $I$ exists. For all $x \in \partial I$, denoting $R_2 = \{z\in I : x-z \cdot (\frac{x-p}{\lVert x-p \rVert} > r/4)$, $|R_2 \cap A|\geq k/(\ln k)^{2}$. We will show that for any $z$ within $\Delta r$ of the boundary of $I$, $\mathbb{P}(F(z; A) > C)< 1/n^3$.

Fix $x$, and assume without loss of generality that $p = [r,0,0,\dots, 0]$ and $x = [0,0,\dots, 0]$

Let $z = [z_1,0,0,\dots,0]$ where $-\sigma\log n < z_1 < \Delta r = O(r/\ln k)$. Then, $\max_{y \in R_1} \frac{z_1 - y_1}{\sigma^2}g(y,z) = \frac{z_1}{\sigma^2}\exp\left(\frac{-z_1^2}{2\sigma^2}\right)< \frac{\Delta r}{\sigma^2}\exp\left(\frac{-(\Delta r)^2}{2\sigma^2}\right) = \frac{\Delta r}{\sigma^2}(1-o(1))$ Also, separately taking the minima of $y_1 - z_1$ and $g(y,z)$, $\min_{y \in R_2} \frac{y_1 - z_1}{\sigma^2}g(y,z) = \frac{r}{4\sigma^2}(1-o(1))$. Bounding the derivative again:
\begin{equation}
\frac{\partial}{\partial z_1}E[z]\geq -|A \cap R_1| \frac{\Delta r}{\sigma^2}(1-o(1)) + |A \cap R_2|\frac{r}{4\sigma^2}(1-o(1)) - |A\setminus I| \frac{O(1)}{\sigma}
\end{equation}
By the assumption of the case, $|A \cap R_2| \geq k/(\ln k)^2$. Also, by the assumption of the theorem $|A\setminus I| = O(k^{2/3})$. Replacing this:
$$\frac{\partial}{\partial z_1}E[z]\geq -k \frac{\Delta r}{\sigma^2}(1-o(1)) +\frac{kr}{4\sigma^2(\ln k)^2}(1-o(1)) -\frac{O(k^{2/3})}{\sigma}$$

Define $\Delta r = r/(\ln k)^3$.
$$\frac{\partial}{\partial z_1}E[z]\geq -k(1-o(1)) \frac{r}{\sigma^2 (\ln k)^3}+\frac{\Theta(k)r}{\sigma^2(\ln k)^{2}} - \frac{O(k^{2/3})}{\sigma}=\frac{kr}{\sigma^2 (\ln k)^{3}}\left[1-o(1) + \Theta(\ln k)\right]$$

Suppose $z = [z_1, 0, 0, \dots, 0]$ where $z_1 < \Delta r/2$. Let $w = z_1 + [\Delta r/2, 0, 0, \dots, 0]$. Using the lower bound on the derivative, 

$$E[w] \geq E[z] + \frac{\Delta r}{2}\cdot\frac{kr}{\sigma^2 (\ln k)^{3}}\left[1-o(1) + \ln k\right] \geq E[z] + \frac{1}{4}\frac{kr^2}{\sigma^2}\ln k$$
Using the same bound as above for $k-E[z]$: 
$$\hat k - E[z] \leq \hat k(1 - e^{-2r^2/\sigma^2})\leq k \frac{2r^2}{\sigma^2}$$
We can apply Lemma \ref{lem-chernoff} to $k-E[z']$; let $Z = k-F(z'; A)$. Then, $\mathbb{P}(Z < (1-\delta)\mathbb{E}Z])\leq e^{-\delta^2 \mathbb{E}Z/2}$. So, 
$$\mathbb{P}(Z <\mathbb{E}Z - \Theta(1)\mathbb{E}Z) < e^{-\Theta(1) k^{1/3+2\epsilon}} $$
The probability that there exist $k^{2/3}$ points which violate the condition is at most:

$$ {n \choose k^{2/3}}(e^{-\Theta(1) k^{1/3+2\epsilon}})^{k^{2/3}} < e^{k^{2/3}\log n}e^{-\Theta(1)k^{1+2\epsilon}}$$

Since there are at most ${n \choose k} = O(e^{k \log n})$ possible $k$-subsets of $V$, this is true by the union bound for all subsets $A$ with high probability.

\end{proof}

\section{Continuous $\alpha$-cap process}
\label{sec:cts}
This section considers a continuous analog of the $k$-cap process. To understand the connection, one can imagine a graph with infinite nodes whose hidden variables span a subspace of $\mathbb{R}$. Rather than choosing a fixed $k$ vertices to fire, a constant fraction $\alpha$ of this subspace is activated. 

For clarity, the definition for the $\alpha$-cap process in one dimension is restate below. We assume that the hidden variables are drawn from $[0, 1]$; however, the analysis will be similar for any finite interval. 
\begin{definition}[$\alpha$-cap Process in 1-D]
	\label{alpha-cap-1D}
	Let $A_0$ be a finite union of intervals on [0,1]. Let $\alpha = |A_0|$ and for an integrable function $g:[0,1]\rightarrow \mathbb{R}$,  
	let
	\begin{equation}
	F_t(x) = \int_{0}^{1} A_t(y)g(y-x)\, dy
	\end{equation}
	Then, define the next step
	\[ A_{t+1}(x)=\begin{cases} 
	0 & F_t(x)<C_{t} \\
	1 & F_t(x)\geq C_t\\
	\end{cases}
	\]
	where $C_t \in [0,1]$ is the solution to $\int_0^1 A_{t+1}(x) dx = \alpha$.
\end{definition}
The goal of this section will be to show that this process converges to a single interval, and that the convergence time depends on properties of $g$ and $g'$. 

Two motivating examples of the function $g$ are (proportional to) the Gaussian density with variance $\sigma^2$, 
$g(x)=\exp(\frac{-x^2}{2\sigma^2})$
and the inverse square distance $g(x) = 1/(1+x^2)$. 

First,  we will prove that single intervals of width $\alpha$ are the only possible fixed points. 

\begin{theorem}[Fixed Points]
	\label{gauss_fixed_pt}
	For any even, nonnegative, integrable function $g:[0,1]\rightarrow \Re_+$ with $g'(x) < 0$ for all $x > 0$, 
	the only fixed points ($A_{t+1}=A_{t}$) of the $\alpha$-cap Process are single intervals of width $\alpha$. 
\end{theorem}
The next lemma follows from the properties of $g$.
\begin{lemma}\label{lem:sym}
	The following holds for all $b>a$:
	\begin{equation*}
	\int_{a}^{b}g(y-a)\, dy=\int_{a}^{b} g(y-b)\, dy\\
	\end{equation*}
	
\end{lemma}
We proceed to the proof of the fixed point characterization.
\begin{proof}[Proof of Theorem~\ref{gauss_fixed_pt}]
	First, we show that if $A_t= [a, b]$ is a single interval, then it is a fixed point.  Since $A_t$ is 1 on the interval and 0 elsewhere, we can rewrite $F_t(x)$:
	\begin{equation}
	F_t(x)=\int_{a}^{b} g(y-x)\, dy
	\end{equation}
	Let $C_t=F_t(a)=F_t(b)=\int_{a}^{b} g(y-a)\, dy$. If $x\in[a, b]$,
	\begin{equation*}
	F_t(x)=\int_{a}^{b}g(y-x)\, dy\geq \int_{a}^{b} g(y-a)\, dy=C_t
	\end{equation*}
	It's easiest to see this by breaking it into two integrals. The first is
	\begin{equation*}
	\int_{a}^{x}g(y-x)\, dy=\int_{a}^{x} g(y-a)\, dy\\
	\end{equation*}
	which holds by Lemma~\ref{lem:sym}, and the second is
	\begin{equation*}
	\int_{x}^{b}g(y-x)\, dy\geq \int_{x}^{b} g(y-a)\, dy
	\end{equation*}
	which holds because $a<x$, which means $g(y-x)>g(y-a)$ for all $y\in[x, b]$. \\
	Therefore, $F_t(x)\geq F_t(a)=C_t$ for all $x\in[a, b]$. \\
	Similarly, if $x<a$ or $x>b$, 
	\begin{equation*}
	F_t(x)=\int_{a}^{b}g(y-x)\, dy<\int_{a}^{b} g(y-a)\, dy=C_t
	\end{equation*} 
	
	This implies that if $C_t$ is chosen in this way, then $A_{t+1}=[a, b]=A_t$.

	Next, let $A_t$ be the union of finite intervals. Let $A_t=\bigcup_{j=1}^{n}[a_j, b_j]$ where for all $j< n$, $a_j<b_j<a_{j+1}<b_{j+1}$, and $n>1$. We will show that this is not fixed. 
	
	$F_t(x)$ can be expressed as the following:
	\begin{equation*}
	F_t(x)=\sum_{j=1}^{n}\int_{a_j}^{b_j}g(y-x)\, dy
	\end{equation*}
	Consider $F_t(a_n)$ and $F_t(b_n)$.
	\begin{equation*}
	F_t(a_n)=\int_{a_n}^{b_n}g(y-a_n)\, dy + \sum_{j=1}^{n-1}\int_{a_j}^{b_j}g(y-a_n)\, dy
	\end{equation*}
	\begin{equation*}
	F_t(b_n)=\int_{a_n}^{b_n}g(y-b_n)\, dy + \sum_{j=1}^{n-1}\int_{a_j}^{b_j}g(y-b_n)\, dy
	\end{equation*}
	By Lemma~\ref{lem:sym}, $\int_{a_n}^{b_n}g(y-a_n)\, dy=\int_{a_n}^{b_n}g(y-b_n)\, dy$. Also, since $a_n<b_n$, $\int_{a_j}^{b_j}g(y-a_n)\, dy>\int_{a_j}^{b_j}g(y-b_n)\, dy$ for all $j<n$. Therefore, $F_t(a_n)>F_t(b_n)$. \\
	Given this, there is no value of $C_t$ where $F_t(x)$ is greater for all $x\in [a_j, b_j]$ and less for all $x\notin A_t$. Notice that by definition, if $A_{t+1}=A_t$, $C_t\leq F_t(x)$ for all $x\in A_{t+1}$. In particular, $C_t\leq F_t(b_n) < F_t(a_n)$. However, since $F_t(x)$ is a continuous function, $C_t<F_t(x)$ in some small window $[a_n-\epsilon, a_n]$. By definition, $a_n-\epsilon \in A_{t+1}$, but $a_n-\epsilon \notin A_t$ for small enough $\epsilon$. This implies $A_t \neq A_{t+1}$. 
\end{proof}
Now, we will prove the main convergence theorem, which we originally introduced in Section \ref{sec:cts}. 
\thmcts*

\begin{proof}[Proof of Theorem~\ref{thm:cts_convergence}]
	At a given step $t \geq 0$, $A_t$ is a union of finite intervals on [0,1]. We will show that if the number of intervals is greater than 1, the distance between the midpoints of the first and last intervals decreases at each step, and this decrease is not diminishing. \\
	Let $A_t=\bigcup_{j=1}^{n}[a_j, b_j]$ where the intervals are disjoint and increasing; for all $j< n$, $0\leq a_j<b_j<a_{j+1}<b_{j+1}\leq 1$. Define the midpoint of the kth interval $m_k=\frac{a_k+b_k}{2}$. \\
	By~\ref{gauss_fixed_pt}, if $n=1$, then the process has converged (i.e. $A_{t+1}=A_t$). If $n>1$, $[a_1, b_1]$ and $[a_n, b_n]$ are the first and last intervals in $A_t$, respectively. We will show that the distance between the midpoints, $m_n-m_1$, decreases by at least a constant. \\
	
	Suppose the first local maximum of $F_t$ occurs at a value $m_1 + \delta$. This proof will show that the shift in the midpoint $m_1$ is bounded from below by a constant depending on $\delta$. Then, it will show that if $[a_1, b_1]$ is not large  ($b_1-a_1<\frac{\alpha}{2}$), $\delta$ is also bounded from below. If $[a_1, b_1]$ is large, a symmetric argument shows that $m_n$ must shift.\\
	
	Recall the definition of $F_t$ and write its derivative:
	\begin{align*}
	F_t(x) &= \sum_{k=1}^{n} \int_{a_k}^{b_k}g(y-x)\, dy\\
	\frac{dF_t}{dx} &= \sum _{k=1}^{n}g(a_k-x)-g(b_k-x)\\
	\end{align*}
	
	Since $g(x)$ decreases with $|x|$, if $x<m_k$, $g(a_k-x)> g(b_k-x)$. This implies $F_t$ is increasing on [0,$m_1$), so the first local maximum must occur at $m_1+\delta > m_1$. \\
	Let $z$ be the minimum value where $F_t(x)\geq C_t$. If $z\geq m_1$, it is simple to show that the shift in $m_1$ depends on $\delta$.
	\begin{itemize}
		\item If $z\geq m_1+\delta$, then $A_{t+1} \cap [0, m_1+\delta]=\emptyset$. Therefore, the midpoint of the first interval is greater than $m_1+\delta$.
		\item If $ m_1\leq z<m_1+\delta$, then $[z, m_1+\delta] \subset A_{t+1}$. The midpoint of the first interval is greater than $\frac{z+m_1+\delta}{2} \geq m_1+\frac{\delta}{2}$. 
		\item If $z<m_1$, the proof is more involved. Let $z=m_1 - \epsilon$ for an $\epsilon > 0$.
	\end{itemize}
	
	Note that it is possible that $z=m_1-\epsilon<a_1$, such that the left end of the interval decreases. However, the midpoint of the interval will always increase. The influence of $F_t$ from the first interval is the same for $m_1-\epsilon$ and $m_1+\epsilon$.
	
	\begin{align*}
	\int_{a_1}^{b_1} g(y-(m_1-\epsilon))\, dy &=\int_{-m_1+\epsilon}^{m_1+\epsilon} g(z)\, dz\\
	&=-\int_{m_1-\epsilon}^{-m_1-\epsilon} g(-z)\, dz\\
	&=\int_{-m_1-\epsilon}^{m_1-\epsilon} g(z)\, dz =\int_{a_1}^{b_1} g(y-(m_1+\epsilon))\, dy
	\end{align*}
	For any $[a_k, b_k]$ where $m_1<a_k<b_k$, 
	\begin{align*}
	\int_{a_k}^{b_k} g(y-(m_1+\epsilon))\, dy &=\int_{a_k-m_1-\epsilon}^{b_k-m_1-\epsilon} g(z)\, dz \\
	&>\int_{a_k-m_1+\epsilon}^{b_k-m_1+\epsilon} g(z)\, dz\\
	&=\int_{a_k}^{b_k} g(y-(m_1-\epsilon))\, dy
	\end{align*}
	This implies that $F_t(m_1+\epsilon) > F_t(m_1-\epsilon)$ for any $\epsilon > 0$.  Since $F_t$ is continuous, there is a small value $\epsilon'>0$ such that $F_t(m_1+\epsilon+\epsilon')=C_t$, and $F_t(x)>C_t$ in between. At $A_{t+1}$, the first interval becomes $[m_1-\epsilon, m_1+\epsilon+\epsilon']$, which has the midpoint $m_1+\frac{\epsilon'}{2}$. Therefore, the midpoint of the first interval increases.\\
	\begin{figure}[h]
		\includegraphics[width=0.5\textwidth]{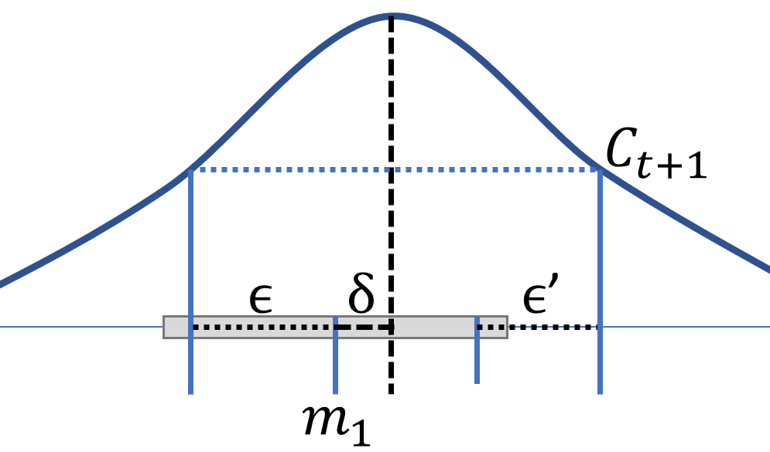}
		\centering
		\caption{An illustration of the terms defined in this proof. The grey box represents the first interval, $[a_1, b_1]$. The curve is $F_t$.}
		\label{fig:defs}
	\end{figure}
	
	Next, we will show that the $\epsilon'$ is bounded below by a constant factor of $\delta$. 
	Either $[a_1, b_1]$ or $[a_n, b_n]$ must be smaller than $\frac{\alpha}{2}$. Assume that $b_1-a_1 \leq \frac{\alpha}{2}$; if not, a symmetric argument applies to the last interval. 
	\begin{claim}
		\label{claim-eps}
		Assume that $b_1-a_1 < \frac{\alpha}{2}$. Then, $$F_t(m_1+\epsilon)- F_t(m_1-\epsilon)\geq \frac{\alpha}{2}\min\{\frac{\alpha}{4}, \epsilon\}\min_{y \in \{\frac{\alpha}{4}, 2\}}|g'(y)|$$
	\end{claim}
	\begin{proof}
		Let $m$ be the median of $A_t$; by the assumption, $m\notin[a_1, b_1]$, so $m \geq a_2$. 
		\begin{align*}
		F_t(m_1+\epsilon)- F_t(m_1-\epsilon) &= \sum_{k=1}^{n}\int_{a_k}^{b_k} g(y-(m_1+\epsilon))\, dy- \sum_{k=1}^{n}\int_{a_k}^{b_k} g(y-(m_1-\epsilon))\, dy\\ 
		&=\sum_{k=2}^{n}\int_{a_k}^{b_k} g(y-(m_1+\epsilon))-g(y-(m_1-\epsilon))\, dy\\
		&\geq \sum_{k=2}^{n} (b_k-a_k)[\min_{y\in[a_k, b_k]}g(y-(m_1+\epsilon))-g(y-(m_1-\epsilon))]\\
		&\geq \frac{\alpha}{2}[\min_{y\in[m, b_n]}g(y-(m_1+\epsilon))-g(y-(m_1-\epsilon))]\\
		\end{align*}
		By definition, $y-m_1\geq m-m_1>\frac{\alpha}{4}$; in the case where $m-(m_1+\epsilon) > 0$,  $$g(y-m_1-\epsilon)-g(y-m_1+\epsilon) > g(y-m_1)-g(y-m_1+\epsilon)\geq \epsilon \min_{z\in[0, \epsilon]}|g'(y-m_1+z)|$$ 
		
		If $\epsilon$ is large enough such that $y-m_1-\epsilon < 0$ for some $y$, then by the symmetry of $g$, $g(y-m_1-\epsilon) = g(m_1+\epsilon-y)$. Therefore, 
		\begin{align*}
		g(y-m_1-\epsilon)-g(y-m_1+\epsilon) &= g(m_1+\epsilon-y) - g(y-m_1+\epsilon) \geq g(\epsilon)-g(y-m_1+\epsilon) \\
		&\geq [y-m_1]\min_{z\in[0, y-m_1]}|g'(\epsilon+z)|
		\end{align*}
		In both cases, $F_t(m_1+\epsilon)- F_t(m_1-\epsilon)$ by restricting $g'$. Let $c_1=\min_{y \in [\frac{\alpha}{4}, 2]}|g'(y)|\leq\min_{y \in [m-m_1, b_n+\epsilon-m_1]}|g'(y)|$. Since $g'$ is strictly decreasing, Therefore, for $\epsilon$ small ($\epsilon < m-m_1$): 
		\begin{align*}
		F_t(m_1+\epsilon)- F_t(m_1-\epsilon) &\geq \frac{\alpha}{2}[\min_{y\in[m, b_n]}g(y-(m_1+\epsilon))-g(y-(m_1-\epsilon))]\\
		&\geq \frac{\alpha\epsilon}{2} \min_{y\in[m, b_n]} \min_{z\in[0, \epsilon]}g'(y-m_1+z) \\
		&\geq \frac{\alpha\epsilon}{2} \min_{y\in[m, b_n+\epsilon]} g'(y-m_1)\\
		&\geq \frac{\alpha}{2} \epsilon c_1
		\end{align*}
		
		For $\epsilon$ large ($\epsilon \geq m-m_1$): 
		\begin{align*}
		F_t(m_1+\epsilon)- F_t(m_1-\epsilon) &\geq \frac{\alpha}{2}[\min_{y\in[a_2, b_n]}g(y-(m_1+\epsilon))-g(y-(m_1-\epsilon))]\\
		&\geq \frac{\alpha}{2} \min_{y\in[m, b_n]} [y-m_1]\min_{z\in[0, y-m_1]}|g'(\epsilon+z)|\\
		&\geq  \frac{\alpha}{2}\frac{\alpha}{4}\min_{z\in[0, b_n-m_1]}|g'(\epsilon+z)|\\
		&\geq \frac{\alpha}{2}\frac{\alpha}{4} c_1
		\end{align*}
		
		In this case, $F_t(m_1+\epsilon)- F_t(m_1-\epsilon)\geq \frac{\alpha}{2}\min\{\frac{\alpha}{4}, \epsilon\}c_1$. \\
	\end{proof}
	
	Let $c_2=g(0)-g(1)$. Since $g$ is continuous and decreasing, $c_2>0$ and $|\frac{dF_t}{dx}| =  |\sum _{k=1}^{n}g(a_k-x)-g(b_k-x)|\leq c_2$ . By~\ref{claim-eps}, 
	$$F_t(m_1+\epsilon)- F_t(m_1+\epsilon+\epsilon')\geq \frac{\alpha}{2}\min\{\frac{\alpha}{4}, \epsilon\}c_1$$
	and 
	$$F_t(m_1+\epsilon)- F_t(m_1+\epsilon+\epsilon') \leq c_2\epsilon'$$
	Combining these two equations implies, 
	$$\epsilon' \geq \frac{\alpha}{2}\min\{\frac{\alpha}{4}, \epsilon\}\frac{c_1}{c_2}$$
	The only unbounded value in this equation is $\epsilon$. Recall that $m_1 + \delta$ is defined to be the earliest local maximum of $F_t$; in the case of small $\epsilon$, $[m_1-\epsilon, m_1+\delta]$ is a subset of the first interval. If $\epsilon \leq \frac{\delta}{2}$, then the midpoint is at least $m_1+\frac{\delta}{4}$.\\
	
	Therefore, the shift of the midpoint of the first interval depends on the location of the first local maximum of $F_t$. We will show that this is also bounded from below by a constant value.
	\begin{claim}
		\label{claim-delta}
		\begin{equation*}
		\delta \geq \frac{\min_{y \in [\frac{\alpha}{8}, 1]}|g'(y)|}{\max_{z\in[0, 1]}|g'(z)|}\cdot\frac{\alpha}{8}
		\end{equation*}
	\end{claim}
	\begin{proof}
		By assumption, at $m_1+\delta$, $\frac{dF_t}{dx} = 0$. \\
		
		\begin{align*}
		&\sum_{j=1}^{n}g(a_j-(m_1+\delta))-g(b_j-(m_1+\delta)) = 0\\
		\end{align*}
		Since $g$ is decreasing and symmetric, the sign of $g(a_j-(m_1+\delta))-g(b_j-(m_1+\delta))$ depends on whether $m_1 + \delta$ is closer to $a_j$ or $b_j$. Suppose this term is negative for $j=1,...,k$ and positive for $j=k+1,...,n$. Additionally, assume that $\delta < \frac{\alpha}{8}$; since $b_1-a_1<\frac{\alpha}{2}$, $m_1+\delta < a_1 + \frac{\alpha}{4} + \frac{\alpha}{8}<m$. (Recall that $m$ is defined as the median of $A_t$). Using this, we can assume that $a_{k+1}\leq m$.\\
		
		Again, since $ g $ is decreasing, $g(a_j-x)-g(b_j-x)\geq (b_j-a_j)\min_{y \in [a_j, b_j]}|g'(y-x)|$. \\	
		\begin{align*}
		\sum_{j=k+1}^{n}g(a_j-(m_1+\delta))-g(b_j-(m_1+\delta)) &\geq \sum_{j=k+1}^{n} \min_{y \in [a_j, b_j]}|g'(y-(m_1+\delta)|(b_j-a_j) \\
		&\geq \min_{y \in [m, 1]}|g'(y-(m_1+\delta)|\frac{\alpha}{2}\\
		&\geq\min_{y \in [\frac{\alpha}{8}, 1]}|g'(y)|\frac{\alpha}{2}\\
		\end{align*}	
		For $j=1...k$, $g(a_j-(m_1+\delta))-g(b_j-(m_1+\delta))<0$. 
		\begin{align*}
		\sum_{j=1}^{k}g(b_j-(m_1+\delta))-g(a_j-(m_1+\delta))&\leq g(b_k-(m_1+\delta))-g(a_1-(m_1+\delta))\\
		&\leq g(b_k-(m_1+\delta))-g(a_1-(m_1+\delta))\\
		&\leq (|a_1-(m_1+\delta)|-|b_k-(m_1+\delta)|)\max_{z\in[0, 1]}|g'(z)|\\
		&\leq (2(m_1+\delta) -b_k-a_1))\max_{z\in[0, 1]}|g'(z)|
		\end{align*} 
		
		Recall that since $\frac{dF_t}{dx} = 0$, we have:\\
		\begin{equation*}
		\sum_{j=1}^{k}g(b_j-(m_1+\delta))-g(a_j-(m_1+\delta))=\sum_{j=k+1}^{n}g(a_j-(m_1+\delta))-g(b_j-(m_1+\delta))
		\end{equation*}
		Combining the two equations above:\\
		\begin{equation*}
		(2(m_1+\delta) -b_k-a_1))\max_{z\in[0, 1]}|g'(z)| \geq \min_{y \in [\frac{\alpha}{8}, 1]}|g'(y)|\frac{\alpha}{2}
		\end{equation*}
		Therefore, $\delta$ is bounded:
		\begin{equation*}
		\delta \geq \frac{\min_{y \in [\frac{\alpha}{8}, 1]}|g'(y)|}{\max_{z\in[0, 1]}|g'(z)|}\cdot\frac{\alpha}{4} + \frac{a_1+b_k}{2}-m_1
		\end{equation*}

		Since $b_k \geq b_1$, the midpoint of $a_1$ and $b_k$ is greater than $m_1$. Combining this fact with the earlier assumption that $\delta < \frac{\alpha}{8}$, we have
		\begin{equation*}
		\delta \geq \frac{\min_{y \in [\frac{\alpha}{8}, 1]}|g'(y)|}{\max_{z\in[0, 1]}|g'(z)|}\cdot\frac{\alpha}{8}
		\end{equation*} 
	\end{proof}
	
	This implies that in the case where $b_1-a_1 \leq \frac{\alpha}{2}$, the midpoint of the first interval increases by a constant value. By a symmetric argument, if $b_1-a_1 > \frac{\alpha}{2}$, then $b_n-a_n < \frac{\alpha}{2}$, and the midpoint of the last interval shifts decreases by a constant value. Since $m_n-m_1$ decreases at each step, the process must converge to a single interval in finite steps.

\end{proof}

This theorem indicates that the speed of convergence depends on the function $g$ and the size of $A_0$ $(\alpha)$. 
This process converges to a fixed point in at most 
\[
O\left(\frac{\max_{[0,1]}|g'(x)|}{\min_{[\frac{\alpha}{8},1]}|g'(x)|}\right)
\]
steps. For example, consider the Gaussian function $g(x)=\exp(\frac{-x^2}{2\sigma^2})$. We have $g'(x)=-\frac{x}{\sigma^2}\exp(\frac{-x^2}{2\sigma^2})$. The maximum is:
\[
\max_{[0,1]}|g'(x)| = |g'(\sigma)| = \frac{1}{\sigma}\exp(-1/2)\\
\]
The minimum can occur at either endpoint depending on $\alpha$ and $\sigma$:
\[	
\min_{[\frac{\alpha}{8},1]} |g'(x)|= \min\left(\left|g'\left(\frac{\alpha}{8}\right)\right|,\left|g'(1)\right|\right) \approx \min\left(\frac{\alpha}{8\sigma^2}, \frac{1}{\sigma^2}\exp\left(\frac{-1}{2\sigma^2}\right)\right)
\]
\\
If $\alpha < 8\left(\frac{-1}{2\sigma^2}\right)$, the lower bound on the number of convergence steps is 
\[
O\left(\frac{\frac{1}{\sigma}\exp(-1/2)}{\frac{1}{\sigma^2}\exp\left(\frac{-1}{2\sigma^2}\right)}\right)\approx O(\sigma)
\]\\
Otherwise, the bound is 
\[
O\left(\frac{\frac{1}{\sigma}\exp(-1/2)}{\frac{\alpha}{8\sigma^2}}\right)\approx O(\frac{\sigma}{\alpha})
\]\\
Another example is the inverse square distance $g(x) = \frac{1}{c+x^2}$ for a value $c>0$. We have $g'(x)=\frac{-2x}{(c+x)^2}$. The maximum occurs at $\max_{[0,1]}|g'(x)|=\frac{-3\sqrt{3}}{8}c^{-3/2}$. The minimum occurs at one of the endpoints: $$\min_{[\frac{\alpha}{8},1]}|g'(x)|=\min\{|g'(\alpha/8)|, |g'(1)|\}\approx \min\left(\frac{\alpha}{4c^2+c\alpha}, \frac{2}{(c+1)^2}\right)$$
Therefore, the lower bound on the number of convergence steps when $c$ is not small is approximately:
\[
O\left(\frac{c^{-3/2}(c^2+c\alpha)}{\alpha}\right)= O\left(\frac{c^{1/2}}{\alpha} + c^{-1/2}\right)
\]
For sufficiently small $c$, this gives:
\[
O\left(c^{-3/2}(c+1)^2\right) = O(c^{-3/2})
\]

\section{Conclusion and further questions}
\paragraph{Plasticity.} Our proof shows that plasticity is not necessary for the convergence of the $k$-cap mechanism. Previous analysis of this process on random graphs studied a variant of the problem where edges were given a weight, initially set to 1. If two neighboring vertices fired consecutively, the weight of their edge was boosted by a factor of $1+\beta$. In Erdős–Rényi random graphs, this weight proved to be vital for convergence; it allowed a set of vertices to become associated over time, causing them to fire together~\cite{papadimitriou2020brain, dabagia2021assemblies}. We have shown that, given sufficiently local graph structure, it is possible for the process without plasticity to converge to a subset which is small compared to $n$.
\paragraph{Parameter range.}
As mentioned in the introduction, $\sigma = O(1/k^{1/d})$ is the parameter range where the concentration behavior in step $t=0$ seems to emerge. However, it is unclear whether this is a true 'threshold'; it would be interesting to determine whether $\sigma = k^{-1/d-\epsilon}$ or $\sigma = k^{-1/d+\epsilon}$ behaves like a pure random graph, or whether some weaker convergence behavior emerges. 
\paragraph{General edge probability functions.}
In this paper, we have focused our attention on graphs whose edge probability is proportional to the Gaussian function, $g(x; \sigma) = e^{-x^2/2\sigma^2}$. Our proof exploits the structure of this function; it is important that the edge probability drops off exponentially after a certain distance. One possible direction for future research is to consider alternate edge probability functions, and see if the behavior of the model deviates significantly. 
\paragraph{Further motivation from Neuroscience.}
As discussed above, the geometric model embeds the nodes of the graph as points in space, and it strongly prefers to connect nodes which are close to each other. Many real-world graphs have a spatial component and a cost associated with long-range connections, so the geometric graph model has theoretical guarantees which match empirical properties of graphs in many domains. One such property is the  clustering coefficient, which measures the prevalence of cliques between the immediate neighborhood of the vertex~\cite{boguna2003class}. In the graph model we have discussed thus far, the clustering coefficient is quite high; in fact, within a small neighborhood of any vertex the probability that the vertices form a  clique is exponentially likely. In particular, high clustering between neurons has been observed in the brain~\cite{SongETAL:05}. 
Of course, the model we have studied is simplified, and it lacks many graph-theoretic structures which have been observed in the brain. There may be interesting algorithmic insights which can be gleaned by mimicking empirically observed structures. Two relevant properties are the \textit{power-law degree distribution} and the \textit{small world property}~\cite{bullmore2009complex}. The first property implies that their are a small set of 'hub' neurons with very high degree (in the geometric random graph, the degree distribution is fairly uniform). The second implies that for any two neurons, the length of the shortest path between them is not very large. 
Both of these properties could have interesting implications for the $k$-cap mechanism. One concrete question is whether, in a graph with a power-law degree distribution, the $k$-cap mechanism is likely to converge to a set of vertices with high degree. 

\paragraph{The continuous model in higher dimension.}
An interesting extension of our continuous model and analysis would be to consider the process in a higher dimensional vector space. Even in $\mathbb{R}^2$, where $g$ is inversely proportional to the euclidean distance between two points, the behavior seems significantly more complicated than the case in $\mathbb{R}$. One open question is to characterize the fixed points (i.e., the sets where $A_{t+1} = A_{t}$) in higher dimensional vector spaces. 
\paragraph{Simulations.} A simulation of the discrete $k$-cap process can be accessed through our GitHub repository~\cite{Reid_Simulations_for_the_2022}. 

\paragraph{Acknowledgements.} The authors are deeply grateful to Christos Papadimitriou, Max Dabagia, Debankur Mukherjee and Jai Moondra for helpful comments and discussions. This work was supported in part by NSF awards CCF-1909756, CCF-2007443 and CCF-2134105.

\printbibliography

\section{Appendix}

\subsection{Probability Preliminaries}
\label{sec:prob-prelims}

The following lemma relates to the distribution of uniform random points in $[0,1]$. It will be referred to frequently throughout the proof.

\begin{lemma}
    \label{lem-n-uniform-random}
    All balls of radius $\sqrt{d/2}\left[\frac{6\log n}{n}\right]^{1/d}$ contain at least one vertex of $G$ almost surely.
\end{lemma}
\begin{proof}
Consider dividing $[0,1]^d$ into $n/(3\log n)$ boxes with side length $[(3\log n)/n]^{1/d}$. 

For any box, the probability that it receives no points of $G$ is $(1-(3\log n)/n)^{n}\leq e^{-3\log n} = n^{-3}$. There are $n/3\log n$ boxes, so by the union bound, the probability that all boxes have at least one point of $G$ is $(2n^2\log n)^{-1}$.

A ball of radius $\sqrt{d/2}\left[\frac{6\log n}{n}\right]^{1/d}$ contains a box of side length $\frac{6\log n}{n}$. Any such box contains at least one box of the partition of $[0,1]^d$. Thus, all balls of this radius contains a vertex of $G$ almost surely. 
\end{proof}

	\begin{lemma}[Balls into Bins]
	\label{lem-balls-into-bins}
	    Suppose $m$ balls are assigned uniformly at random to $n$ bins, where $\frac{n}{polylog(n)}\leq m << n\log n$. Then, with probability $1-o(1)$, the maximum load is at least:
		
		$$\frac{\ln n}{\ln \gamma}\left[1 + 0.9\frac{\ln ^{(2)}\gamma}{\ln \gamma}\right]$$
		
		where $\gamma = \frac{n \log n}{m}$.
	\end{lemma}
 \begin{proof}
 See ~\cite{raab1998balls}
 \end{proof}
\begin{lemma}[Maximum Degree of Geometric Graph]
\label{lem-max-degree}
Let $X = \{x_1, \dots, x_n\}$ be a set of $n$ points chosen uniformly at random on $[0,1]^d$. Define a graph $G(X; r)$ such that there exists an edge between $x_i$ and $x_j$ if $\lVert x_i - x_j\rVert \leq r$.

Define a sequence of radii $(r_n)_n$. Let $\Delta_n$ be the maximum degree of $G$. Define $k_n = \frac{\log n}{\log(\log n/(nr_n^d))}$. 

If $nr_n^d/\log n \rightarrow 0$ and $\log(1/(nr_n^d))/\log(n) \rightarrow 0$ as $n \rightarrow \infty$. Then:
$$ \lim_{n\rightarrow \infty} \frac{\Delta_n}{k_n} = 1 \text{ in probability}$$
and
$$ \liminf_{n\rightarrow \infty} \frac{\Delta_n}{k_n}\geq 1 \text{  Almost surely}$$

\end{lemma}
\begin{proof}
See Theorem 6.10 from~\cite{penrose2003random}
\end{proof}

Next, the following three lemmas contain different tail bounds for the sums of independent indicators.
	\begin{lemma} [Chernoff Bound] 
		Let $X$ be a sum of independent random indicators with mean $\mu$. Then, for any $\delta \geq 0$:
		$$\mathbb{P}(X > (1+\delta)\mu) \leq \exp\left( -\frac{\mu \delta^2}{2+\delta}\right)$$
        $$\mathbb{P}(X < (1-\delta)\mu) \leq \exp\left( -\frac{\mu \delta^2}{2}\right)$$
		\label{lem-chernoff}
	\end{lemma}
	\begin{lemma} 
		For any binomial random variable $X$ with parameters $k, p$, we can bound the probability that it exceeds $M$ for any $M> kp$:
		$$\mathbb{P}(X > M) \leq \exp\left( - k D\left(\frac{M}{k} \mid \mid p\right)\right)$$
		
		where $D\left(a \mid \mid p\right) = a \log \frac{a}{p} + (1-a) \log \frac{1-a}{1-p}$.
		\label{lem-bin-tight}
	\end{lemma}

\begin{lemma}
	\label{lem:pt-lowerbound}
	Let $X= \sum_{i=1}^{k}I_i$ be the sum of $k$ independent indicators with probabilities $\mathbb{P}(I_i) = p_i \in (0, 1)$. Let $\mu = \mathbb{E}X$, and let $t_1, t_2$ be integer values such that $t_1 \geq \lceil \mu \rceil$ and $t_2 > t_1$. Then, 
	
	$$\frac{\mathbb{P}(X\geq t_1)}{\mathbb{P}(X \geq t_2)} > \exp\left(\frac{(t_2-\lceil \mu \rceil)^2 -(t_1-\lfloor \mu \rfloor)^2}{2t_2}\right) > \exp\left(\frac{(t_2-t_1)^2}{t_2}\right)$$
\end{lemma}
\begin{proof}
Duembgen et al~\cite{dumbgen2020density} gives a bound on the ratio of two consecutive probabilities. For any $c$ with $\mathbb{P}(X=c-1)>0$, 
	\begin{equation}
	\label{eq:duembgen}
	\frac{\mathbb{P}(X=c+1)}{\mathbb{P}(X=c)} < \frac{c}{c+1}\frac{\mathbb{P}(X=c)}{\mathbb{P}(X=c-1)}
	\end{equation}
	
	The mode of $X$ is either at $\lfloor \mathbb\mu\rfloor$, $\lceil \mu \rceil$, or is equally attained at both~\cite{darroch1964distribution}. The probablity increases monotonically from $X=0$ up to the mode(s) and then decreases montonically up to $X=k$. 
	
	So, we have for any integer $m > \mu$:
	\begin{equation}
	\label{eq:darroch}
	\frac{\mathbb{P}(X=m+1)}{\mathbb{P}(X=m)}< 1
	\end{equation}
	Using Equation~\ref{eq:duembgen} and~\ref{eq:darroch}, for any integer $s \geq \lceil \mu \rceil$
	\begin{align*}
	\frac{\mathbb{P}(X = s+1)}{\mathbb{P}(X =s)} &< \frac{s}{s+1} \frac{\mathbb{P}(X = s)}{\mathbb{P}(X = s-1)}
	& <\dots
	&<\frac{\lceil\mu\rceil+1}{s+1} \frac{\mathbb{P}(X=\lceil\mu\rceil+1)}{\mathbb{P}(X = \lceil\mu\rceil)}<\frac{\lceil\mu\rceil+1}{s+1}
	\end{align*}
	We can rewrite the ratio of $\mathbb{P}(X = t_2)$ and $\mathbb{P}(X = t_1)$ as the product of ratios with a difference of 1:
	$$\frac{\mathbb{P}(X = t_2)}{\mathbb{P}(X = t_1)}  = \frac{\mathbb{P}(X = t_2)}{\mathbb{P}(X = t_2- 1)}\frac{\mathbb{P}(X = t_2-1)}{\mathbb{P}(X = t_2 - 2)} \cdots \frac{\mathbb{P}(X = t_1+1)}{\mathbb{P}(X = t_1)}$$
	Substituting the bound above:
	$$\frac{\mathbb{P}(X = t_2)}{\mathbb{P}(X = t_1)} < \frac{\left(\lceil\mu\rceil+1\right)^{(t_2-t_1)}}{t_2(t_2-1)(t_2-2) \dots (t_1 + 1)} =\prod_{s=t_1+1}^{t_2}\left(1-\frac{s-\lceil\mu\rceil - 1}{s}\right)$$

	Using the approximation $1-x\leq e^{-x}$, this is at most:
	$$\leq \exp\left(-\sum_{s=t_1+1}^{t_2} \frac{s-\lceil\mu\rceil - 1}{s}\right)\leq \exp\left(-\frac{\sum_{s=t_1}^{t_2} s- \lceil \mu \rceil - 1}{t_2}\right)$$
	Expanding the sum in the numerator:
 \begin{align*}
	&=\exp\left(-\frac{\sum_{s=0}^{t_2-\lceil \mu \rceil - 1} s -\sum_{s=0}^{t_1-\lceil \mu \rceil - 1}s}{t_2}\right) \\
 &= \exp\left(-\frac{(t_2-\lceil \mu \rceil -1)(t_2 - \lceil\mu\rceil) -(t_1-\lceil \mu \rceil -1)(t_1 - \lceil\mu\rceil) }{2t_2}\right) \\
	&\leq \exp\left(-\frac{(t_2-\lceil \mu \rceil)^2 -(t_1-\lfloor \mu \rfloor)^2}{2t_2}\right)\\
\end{align*}
	This ratio decreases as $t_1$ increases and $t_2 - t_1$ remains constant. This means that, for any $i > t_2$, we have $\mathbb{P}(X = i)<\exp\left(-\frac{(t_2-\lceil \mu \rceil)^2 -(t_1-\lfloor \mu \rfloor)^2}{2t_2}\right)\mathbb{P}(X = i - (t_2 - t_1))$
	
	\begin{align*}
	\mathbb{P}(X \geq t_2) &= \sum_{i=\mu + t_2 \sqrt{\mu}}^{k}\mathbb{P}(X = i)\\
	&< \sum_{i=\mu + t_2 \sqrt{\mu}}^{k}\exp\left(-\frac{(t_2-\lceil \mu \rceil)^2 -(t_1-\lfloor \mu \rfloor)^2}{2t_2}\right)\mathbb{P}(X = i - (t_2 - t_1)) \\
	&\leq \exp\left(-\frac{(t_2-\lceil \mu \rceil)^2 -(t_1-\lfloor \mu \rfloor)^2}{2t_2}\right)\mathbb{P}(X \geq t_1)\\
	\end{align*}
	Expanding $(t_2-\lceil \mu \rceil)^2 -(t_1-\lfloor \mu \rfloor)^2$, we get $t_2^2 - t_1^2 - 2\lceil \mu \rceil(t_2 - t_1) = (t_2 - t_1)(t_2 + t_1 - 2\lceil \mu \rceil)$. Since $t_2 > t_1 \geq \lceil \mu \rceil$ by assumption, this exceeds $2(t_2 - t_1)^2$. 
	
\end{proof}

\end{document}